\documentclass[11pt, letterpaper]{amsart}
\usepackage[left=1in,right=1in,bottom=1in,top=1in]{geometry}
\usepackage{graphicx}
\usepackage{amsmath,amsthm,amssymb}
\usepackage{mathrsfs}
\usepackage[all,cmtip]{xy}
\usepackage{cite}
\usepackage{setspace}
\usepackage{color}
\usepackage{url}
\usepackage{bbm}
\usepackage[numbers]{natbib}
\usepackage{ulem,xcolor}

\theoremstyle{plain}
\newtheorem{lem}{Lemma}[section]
\newtheorem{thm}[lem]{Theorem}

\theoremstyle{definition}

\theoremstyle{remark}

\DeclareMathOperator{\Var}{Var}
\DeclareMathOperator{\Ent}{Ent}

\newcommand{\ol}{\overline}
\onehalfspace

\begin{document}

\title{Sublinear variance in Euclidean first-passage percolation}
\begin{abstract}
The Euclidean first-passage percolation model of Howard and Newman is a rotationally invariant percolation model built on a Poisson point process. It is known that the passage time between 0 and $ne_1$ obeys a diffusive upper bound: $\Var T(0,ne_1) \leq Cn$, and in this paper we improve this inequality to $Cn/\log n$. The methods follow the strategy used for sublinear variance proofs on the lattice, using the Falik-Samorodnitsky inequality and a Bernoulli encoding, but with substantial technical difficulties. To deal with the different setup of the Euclidean model, we represent the passage time as a function of Bernoulli sequences and uniform sequences, and develop several ``greedy lattice animal'' arguments.
\end{abstract}
\author{Megan Bernstein$^1$ \and Michael Damron$^2$ \and Torin Greenwood$^3$
}
\maketitle
\let\thefootnote\relax\footnote{$^1~$College of Computing, Georgia Institute of Technology}
\footnote{$^2~$School of Mathematics, Georgia Institute of Technology}
\footnote{$^3~$Department of Mathematics, Rose-Hulman Institute of Technology}
\section{Introduction}

\subsection{Background and main result}

In \cite{Howard:1997}, Howard and Newman introduced the following Euclidean first-passage percolation (FPP) model on $\mathbb{R}^d$: Let $Q \subset \mathbb{R}^d$ be a rate-one Poisson point process.  For any fixed $\alpha > 1$ and any sequence of points (``path'') in $Q$, $r = (r_0, r_1, \ldots, r_k)$, we define the passage time of the path as:
\[
T(r) := \sum_{i = 1}^k \lVert r_i - r_{i - 1} \rVert^\alpha,
\]
where $\lVert \cdot \rVert$ is the Euclidean norm.  The passage time between two points $q, q' \in Q$ is defined as
\[
T(q, q') := \inf_{r, k: r_0 = q, r_k = q'} T(r).
\]
To define the passage time between non-Poisson points, for $x \in \mathbb{R}^d$, let $q(x)$ be the closest point in $Q$ to $x$ in terms of Euclidean distance, with any fixed rule to break ties. Then we set $T(x, y) := T(q(x), q(y))$ for $x,y \in \mathbb{R}^d$.

Euclidean FPP was introduced with the hope that its rotational invariance would help researchers to overcome some of the problems in traditional (lattice) FPP. This invariance is powerful, and allows one to trivially conclude that the limiting shape for the model is a Euclidean ball. Because of this, Howard and Newman were able to verify many long-standing conjectures about the structure of geodesics for FPP models. In the other direction, though, fluctuation bounds for the passage time are more difficult to establish, due to technical difficulties inherent in the Poisson process underlying the model. To date, the best upper bound for the variance is $\Var T(x,y) \leq C\|x-y\|$ from \cite[Theorem~2.1]{Howard:2001}.

In this paper, we continue the study of fluctuations for the Euclidean model, and aim to prove the following sublinear variance bound:
\begin{thm} \label{MainResult}
Let $d \geq 2$ and $\alpha > 1$.  There is a constant $C > 0$ such that for all $\lVert x - y \rVert \geq 1$, $\Var T(x, y) \leq C \frac{\lVert x - y \rVert}{\log \lVert x - y \rVert}$.
\end{thm}

Inequalities for the variance of the form $\|x\|/\log \|x\|$ were first established by Benjamini-Kalai-Schramm \cite{Benjamini:2003} for lattice FPP with Bernoulli edge-weights, and then extended by Bena\"im-Rossignol \cite{Benaim:2008} (along with concentration bounds) to weights in the ``nearly-Gamma'' class. Last, Damron-Hanson-Sosoe \cite{Damron:2014} established such inequalities for general edge-weight distributions. We will follow the strategy of the last two papers, whose main tools are Entropy inequalities and Bernoulli encodings. Whereas in \cite{Damron:2014}, edge-weights were represented in terms of infinite sequences of Bernoulli variables, we will need more variables, representing the Poisson process using both Bernoulli sequences and uniform sequences. We will also need several refinements of so-called ``lattice animal'' inequalities, first used by Howard-Newman, to overcome the difficulties present in the Euclidean model. An outline of the main steps and tools needed appears below in Section~\ref{sec: framework}.

The main motivation for proving a sublinear variance bound in Euclidean FPP comes from estimating the error in the law of large numbers for the passage time. By subadditivity, the ``time constant'' $\mu$ exists and is defined by the formula
\[
\mu = \lim_{n \to \infty} \frac{\mathbb{E}T(0,ne_1)}{n}.
\]
This convergence also holds almost surely, due to the subadditive ergodic theorem, and leads to the heuristic equation $T(0,ne_1) = \mu n + o(n)$. It is customary to break the error term into a random fluctuation and a nonrandom fluctuation term:
\[
T(0,ne_1) - \mu n = (T(0,ne_1) - \mathbb{E}T(0,ne_1)) + (\mathbb{E}T(0,ne_1) - \mu n).
\]
The typical way to express bounds on the random term is by using either a variance upper bound or a concentration inequality. The nonrandom term can be controlled using the random term. Indeed, methods of Howard and Newman (see \cite[Eq.~2.6]{Howard:2001}) show that if one has an inequality of the type
\begin{equation}\label{eq: concentration_need}
\mathbb{P}(|T(0,ne_1) - \mathbb{E}T(0,ne_1)| \geq \lambda \psi(n)) \leq e^{-cn^{\beta'}}
\end{equation}
for constants $\beta,\beta'>0$ and a sufficiently nice function $\psi(n)$ (they use $\psi(n) = \sqrt{n}$), then one can derive the inequality
\[
\mathbb{E}T(0,ne_1)-\mu n \leq C (\log n)^{\beta''} \psi(n),
\]
for yet another $\beta''$. This was improved by Damron-Wang (see \cite[Theorem~2.5]{Damron:2016} and Assumption~2.2 there for conditions on $\psi$) to
\begin{equation}\label{eq: DW}
\mathbb{E}T(0,ne_1) - \mu n \leq C_k (\log^{(k)}n)^{\beta''} \psi(n),
\end{equation}
where $\log^{(k)}$ is the $k$-th iterate of the logarithm, and $C_k$ is a constant depending on $k$. 

One would like to use these tools to prove, as has been done in directed polymers \cite{Alexander:2013}, that the error term $o(n)$ above is actually $o(\sqrt{n})$. To do this, one needs to be able to show \eqref{eq: concentration_need} using $\psi(n) = \sqrt{n/\log n}$. This would effectively bound the random fluctuation term by $o(\sqrt{n})$. Furthermore, applying \eqref{eq: DW} with $k=2$ would give the bound $\sqrt{\frac{n \log \log n}{\log n}}$ for the nonrandom fluctuation term, and this is also $o(\sqrt{n})$. Unfortunately it is not known if \eqref{eq: concentration_need} holds with this choice of $\psi(n)$ (although this is known on the lattice \cite{Benaim:2008, Damron:2014}). Our overall goal is to show this holds in Euclidean FPP, and a first step is to establish our main variance inequality, Theorem~\ref{MainResult}.

\subsection{Proof Framework}\label{sec: framework}
In proving Theorem~\ref{MainResult}, we can immediately reduce to a simpler case: due to statistical isotropy, it suffices to show that $\Var T(0, ne_1) \leq C n/\log n$ for all real $n \geq 1$, where $e_1$ is the unit vector in the first coordinate direction of $\mathbb{R}^d$.

It is crucial for the strategy we present to realize the passage time $T(0,ne_1)$ as a function of independent random variables. To do this, we will break up $\mathbb{R}^d$ into unit boxes, ordered $\mathcal{B}_1, \mathcal{B}_2, \mathcal{B}_3, \ldots$.  In each box $\mathcal{B}_i$, we will define $P_i$ to be a Poisson random variable with parameter $1$ that determines the number of points in the box.  Then, $U_{i, j}$ is a uniformly distributed random variable in $[0, 1)^d$ that determines the location of the $j$th point in the box, if it exists.  Thus, for each box $\mathcal{B}_i$, there is a corresponding outcome space $\Omega_i = \mathbb{N} \times \left(\mathbb{R}^d\right)^\mathbb{N}$ with measure $\mathbb{P}_i$.  If we write an outcome in $\Omega_i$ as $(P_i, U_{i, 1}, U_{i, 2}, U_{i, 3}, \ldots)$, then under the measure $\mathbb{P}_i$, the entries are independent. Finally, for the full space we set $\Omega = \prod_i \Omega_i$, with the probability measure $\mathbb{P} = \prod_i \mathbb{P}_i$.

The proof of Theorem \ref{MainResult} will also require approximations to the distance, $T$, and several probabilistic tools.  These are discussed below.
\subsubsection{Falik-Samorodnitsky inequality}
To begin, we will use the Falik-Samorodnitsky inequality from \cite{Falik:2007}.  Let $\mathcal{F}_i = \sigma\left(\{P_j, U_{j, k}\}: 1 \leq j \leq i, k \geq 1\}\right)$ be the sigma-algebra generated by the variables associated to boxes $\mathcal{B}_1$ through $\mathcal{B}_i$.  Then, for any integrable random variable $Z$ defined on $\Omega$, define
\[
V_i = \mathbb{E}\left[Z | \mathcal{F}_i\right] - \mathbb{E}\left[Z | \mathcal{F}_{i - 1}\right].
\]
Later, we will pick $Z$ to be an averaged version of passage time, $T(0, ne_1)$, but with some modifications. We will denote this as $F_n(0, ne_1)$.  Define the entropy of a non-negative integrable random variable, $X$, as $\Ent(X) = \mathbb{E}\left[X \log(X)\right] - \left[ \mathbb{E} X \right] \cdot \left[\log \mathbb{E}X\right]$. We can use Theorem 2.2 of \cite{Falik:2007} to bound the variance of $Z$ in terms of the entropy of the $V_i$'s:
\begin{thm}[Falik-Samorodnitsky Inequality] \label{FSIneq}
If $\mathbb{E} Z^2 < \infty$, then
\[
\Var Z \cdot \log \left[ \frac{\Var Z}{\sum_i \left(\mathbb{E}|V_i| \right) ^2} \right] \leq \sum_i \Ent V_i^2.
\]
\end{thm}
Given Theorem~\ref{FSIneq}, the strategy of the proof of Theorem~\ref{MainResult} is as follows:  After defining the averaged modified passage time $Z = F_n$ precisely in Section~\ref{sec: avgT}, we will show that the sum $\sum_i(\mathbb{E}|V_i|)^2$ is of order at most $n^{1 - 1/4\alpha}$ in Section \ref{MartDiff} (see Lemma~\ref{SumVBound}).  Then, we will show that $\sum_i \Ent V_i^2$ is at most linear in $n$ in Section \ref{Entropy} (see Lemma~\ref{EntropyBoundGoal}).  From this, we will be able to conclude in Section~\ref{sec: MainProof} that $\Var Z$ is bounded by $Cn/\log n$. It then only remains to show that $\Var T(0,ne_1) \leq \Var F_n  + o(n/\log n)$ (see Lemmas~\ref{FtoT} and \ref{TdoubletoT}) to complete the proof.
\subsubsection{Approximating $T$}
It will be convenient to have an approximation to the passage time $T$ with nice properties.  Following \cite{Howard:2001}, we introduce a modified passage time $T''$, dependent on $n$.  This passage time is first defined on a subset of points $Q_n \subset Q$, defined as follows: divide $\mathbb{R}^d$ into boxes with vertices (corners) at the points $\epsilon/3^n\cdot\left( \mathbb{Z}^d + (1/2, \ldots, 1/2)^d\right)$ for some small $\epsilon > 0$ of the form $\epsilon = 1/k$ for an odd integer, $k$. Choosing $\epsilon$ of this form ensures that the $\epsilon/3^n$ boxes nest evenly within unit boxes.  Further restrictions on $\epsilon$ come from \cite{Howard:2001}, and are described below in Equation \eqref{MEpsilonBound} in the proof of Lemma \ref{OffSegmentBound}, as well as in the statement of Lemma \ref{TDifferences}. Then, $Q_n$ is defined by removing all but the left-most particle of $Q$ in each $\epsilon/3^n$-box (if $Q$ has a particle in the box).  

Next, we define the passage time $T''$ between points in $Q_n$.  This passage time uses a new distance, $\phi_n$, between points in a path:
\[
\phi_n(t) = \begin{cases} 
                t^\alpha, & \text{ if } t \leq h_n \\
		h_n^\alpha + \alpha h_n^{\alpha-1}(t - h_n), & \text{ otherwise.} 
               \end{cases}
\]
Here, $h_0 \geq 1$ and $h_1 \geq h_0$ are large constants with restrictions from \cite{Howard:2001} repeated here before Equation \eqref{SBreakdown} in the proof of Lemma \ref{OffSegmentBound}, and with other restrictions from the statements of Lemmas \ref{E-Regions} and \ref{HN:Regions}.  For all other $n > 0$, $h_n = \max\left(h_0, h_1n^{1/2\alpha}\right)$.  Then, the $T''$-passage time for a path $r = (r_1, r_2, \ldots, r_k)$ of points in $Q_n$ is defined to be:
\[
T''(r) = \sum_{i = 1}^{k - 1} \phi_n(\lVert r_{i + 1} - r_i \rVert).
\]
One advantage of this new distance function is that $\phi_n(\ell)$ grows linearly in $\ell$ for $\ell$ large enough, instead of growing like $\ell^\alpha$. Also $\phi_n$ satisfies an inequality that behaves similarly to a triangle inequality, as reproduced in Lemma \ref{PhiIneq} below.

Finally, we extend the definiton of $T''(a, b)$ to any $a, b \in \mathbb{R}^d$.  Let $\mathfrak{R} = \{r = (r_1, r_2, \ldots, r_{k - 1}): k \geq 1, r_i \in Q_n \mbox{ for all } 1 \leq i \leq k - 1\}$, and define $r_0 = a$ and $r_k = b$.  Then,
\[
T''(a, b) := \inf_{r \in \mathfrak{R}} \sum_{i = 1}^k \phi_n(\lVert r_i - r_{i - 1} \rVert).
\]
In other words, $T''(a, b)$ artificially adds the points $a$ and $b$ into $Q_n$, and then finds the $Q_n$-path between $a$ and $b$ with minimal total length defined by $\phi_n$.

There is some overhead in showing that $T''$ is close enough to $T$ that the variance of $T''$ and $T$ are also close.  This is described in Lemma \ref{TdoubletoT} below, and requires some properties about $T''$ and $T$ from \cite{Howard:2001}.  The properties are summarized in Section \ref{Appendix} below, along with other cited results.
\subsubsection{Averaging $T''$}\label{sec: avgT}
Starting from the Falik-Samorodnitsky inequality, we will need to find a sublinear bound on $\sum_i \mathbb{E} |V_i|^2$, where the $V_i$ are the martingale differences defined above. Unfortunately, it is not known how to prove such an inequality when $Z = T$, or even when $Z=T''$. The problem is that the terms $\mathbb{E} |V_i|^2$ measure a sort of ``influence'' of the variables associated with the box $\mathcal{B}_i$ on $Z$: roughly how much $Z$ changes when these variables are resampled. This influence is not uniform in the location of the box. For instance, if $\mathcal{B}_i$ is close to $0$ or $ne_1$, then the influence will be high, and for other boxes, it is expected to be vanishingly small. This problem occurred in the original sublinear variance proof on the lattice, in which the authors of \cite{Benjamini:2003} noted that they could not even show that the influence of an edge near $(n/2)e_1$ goes to 0 with $n$. This motivated the ``midpoint problem'': show that as $n \to \infty$, the probability that a geodesic from $0$ to $ne_1$ passes through $(n/2)e_1$ goes to 0.

To circumvent the midpoint problem, an averaged passage time was introduced in \cite{Benjamini:2003} and used later in \cite{Benaim:2008}. Their key insight is that if the lattice were replaced by a discrete torus, and the passage time $T$ by a translation invariance ``diameter'' variable, then one could easily show that influences are at most $o(n^{-\epsilon})$. A simpler averaging was later introduced in \cite{Alexander:2013}, and it is this version that we adopt in this paper, applied to the modified time $T''$. Let $\Gamma_n = \left\{z \in \mathbb{Z}^d : \lVert z \rVert_\infty \leq n^{\frac{1}{4\alpha}} \right\}$, where $\alpha$ is the constant from the definition of $T, T'',$ and $\phi_n$.  Any power of $n$ strictly between $0$ and $1/2$ in this definition is compatible with our proof, but choosing $\frac{1}{4\alpha}$ reduces some of our work below in Equation \eqref{DiameterChoice} from the proof Lemma \ref{FtoT}. Let $|\Gamma_n|$ be the number of elements of $\Gamma_n$.  Then, define:
\[
F_n(0, ne_1) := \frac{1}{|\Gamma_n|} \sum_{z \in \Gamma_n} T''(z, z + ne_1).
\] 
This function $F_n$ is the random variable $Z$ that we will substitute into the Falik-Samorodnitsky inequality, Theorem \ref{FSIneq}, as well as in the definition of $V_i$.  The advantage of $F_n$ over $T''$ alone will become apparent in the proof of Lemma \ref{MaxV}.  But, using the averaged version will add a layer of technicalities elsewhere: the proof of Lemma \ref{MaxV} relies on bounds on the distribution of the maximum number of unit boxes touched by a geodesic connecting any two points within $\Gamma_n$ (see Equation \eqref{Gamma-BoxBound}).  We obtain such a bound by using Lemma \ref{BoxBound}.
\subsubsection{Logarithmic Sobolev inequalities}
In order to bound $\sum \Ent V_i^2$ in the Falik-Samorodnitsky inequality, we will use logarithmic Sobolev inequalities to bound entropies by derivatives. Ideally, we would use a logarithmic Sobolev inequality for the uniform random variables encoding the location of the Poisson points in $Q$, and another inequality for the Poisson random variables encoding the number of points in $Q$.  Unfortunately, no such inequality exists for Poisson random variables.  (See the discussion before Equation (5.4) in \cite{Ledoux:1999}.)  To overcome this, we instead encode the Poisson number of points in each box $\mathcal{B}$ in binary as a sequence of Bernoulli random variables.  Then, we can bound the entropies by sums of derivatives with respect to the Bernoulli and uniform random variables corresponding to the number of points and location of the points in each box, $\mathcal{B}$.  This part of the argument appears in Section \ref{Entropy}.
\subsubsection{Greedy lattice animals}
Throughout the paper, we will need to find bounds on expectations of objects like the number of boxes a geodesic passes through, the number of Poisson points in boxes near geodesics, and the sums of lengths of segments on the geodesic.  One way to approach bounds like these is to find bounds on the maximum for any path, regardless of whether it is a geodesic.  More precisely, these variables are of the form $f(P)$, where $P$ is a geodesic, and we often bound them by $\max_P f(P)$, where the maximum is over a class of paths which contains the geodesic. This leads to ``greedy lattice animal'' arguments, like those found in \cite{Cox:1993}, \cite{Gandolfi:1994}, and \cite{Dembo:2001}. Several of the greedy lattice animal bounds have similar proof structures and are useful throughout the paper.  So, these bounds are compiled in Section \ref{GreedyLattice}.

One of the more challenging parts of the proof of Theorem \ref{MainResult} will be bounding the change in the length of a geodesic when points are added to $Q_n$, which is needed to bound the derivatives of $T''$ with respect to the Bernoulli random variables described above.  This is in contrast to removing points in $Q_n$, where the change in geodesic length can be bounded in terms of segments of the geodesic.  In order to bound this difference, we will need to generalize the proof of Equation (3.10) in \cite{Howard:2001} which relied on greedy lattice animals.  We will argue that the only way a geodesic can be substantially changed by adding points is if there is a large region near the geodesic that has no Poisson points.  For further details, see Lemma \ref{OffSegmentBound} and the discussion preceding it.

Throughout our proof, we will make use of many large and small constants.  Where possible, we use capital letters for large constants (like $C_1$) and lower case letters for small ones (like $c_2$).  Constants are not the same in different proofs unless specified.  \\ \ 

\hrule
\section{Bound on $\sum_i \left(\mathbb{E}|V_i| \right) ^2$} \label{MartDiff}
Our goal in this section is to show the following bound:
\begin{lem} \label{SumVBound}
There exists a constant $C > 0$ such that for all $n \geq 1$,
 \[
 \sum_i (\mathbb{E}|V_i|)^2 \leq C n^{1 - \frac{1}{4\alpha}}.
 \]
\end{lem}

In order to prove this, we will show (Lemma~\ref{lin}) that the the sum $\sum_i \mathbb{E} |V_i|$ will grow at most linearly in $n$, while (Lemma~\ref{MaxV}) the maximum of the $\mathbb{E}|V_i|$'s will decay at least as fast as $n^{-\frac{1}{4\alpha}}$.

\begin{lem} \label{lin}
There exists a constant $C$ so that for all $n \geq 1$,
\[ \sum_i \mathbb{E}|V_i| \leq Cn.\]
\end{lem}

\begin{proof}
Recall that $\mathbb{E}|V_i| := \mathbb{E}\big[ |\mathbb{E}[F_n | \mathcal{F}_i] - \mathbb{E}[F_n | \mathcal{F}_{i - 1}]|\big]$, where $F_n := F_n(0, ne_1)$.  Let $\tilde F_{n, i}$ represent the value of $F_n$ when the Poisson and uniform random variables corresponding to box $\mathcal{B}_i$ are resampled.  That is, if $F_n$ is a function of the environment $((P_1, U_{1,1}, U_{1,2}, \dots), (P_2, U_{2,1}, U_{2,2}, \dots), \dots)$, then $\tilde F_{n,i}$ takes the value of $F_n$ evaluated in the environment in which $(P_i, U_{i,1}, U_{i,2}, \dots)$ is replaced by an independent copy $(P_i', U_{i,1}', U_{i,2}', \dots)$. Then, $\mathbb{E}|V_i| \leq \mathbb{E}|\tilde F_{n, i} - F_n|$ by Jensen's inequality.  Indeed, writing $X_1$ for the Poisson configuration in boxes preceding $\mathcal{B}_i$, $X_2$ for the configuration inside $\mathcal{B}_i$, and $X_3$ for the configuration in boxes after $\mathcal{B}_i$, with $\mu_{X_i}$ their respective distributions,
\begin{align}
&\mathbb{E}|V_i| \nonumber \\
=~& \mathbb{E}\left| \int F_n(x_1, x_2, x_3)~\text{d}\mu_{X_3}(x_3) - \int F_n(x_1,x_2,x_3)~\text{d}\mu_{X_2}(x_2)\text{d}\mu_{X_3}(x_3)\right| \nonumber \\
=~& \int\int\left| \int \int \left(F_n(x_1,x_2',x_3)-F_n(x_1,x_2,x_3)\right)~\text{d}\mu_{X_2}(x_2)\text{d}\mu_{X_3}(x_3) \right|~\text{d}\mu_{X_2}(x_2')~\text{d}\mu_{X_1}(x_1) \nonumber \\
\leq~& \int\int\int\int \left| F_n(x_1,x_2',x_3)-F_n(x_1,x_2,x_3)\right|  ~\text{d}\mu_{X_1}(x_1)\text{d}\mu_{X_2}(x_2)\text{d}\mu_{X_2}(x_2')\text{d}\mu_{X_3}(x_3) \nonumber \\
=~& \mathbb{E}|\tilde F_{n,i} - F_n|. \label{eq: v_i_equiv}
\end{align}
We will also let $\tilde T_i''$ represent $T''$ after the points in box $\mathcal{B}_i$ are resampled, and we let $\mathbbm{1}_{\{\tilde T_i'' \geq T''\}}(z)$ be the indicator function of the event that $\tilde T_i''(z,z+ne_1) \geq T''(z,z+ne_1)$.  Then,
\begin{align}
\sum_i \mathbb{E} |V_i| &\leq \sum_i \mathbb{E} |\tilde F_{n, i} - F_n|\nonumber \\
&\leq \frac{1}{|\Gamma_n|} \sum_i \mathbb{E} \left[ \sum_{z \in \Gamma_n} \big|\tilde T_i''(z, z + ne_1) - T''(z, z + ne_1) \big| \right] \nonumber \\
&\leq \frac{2}{|\Gamma_n|} \sum_{z \in \Gamma_n} \sum_i \mathbb{E} \left| \left(\tilde T_i''(z, z + ne_1) - T''(z, z + ne_1)\right) \mathbbm{1}_{\{\tilde T_i'' \geq T''\}}(z) \right|. \label{SumValue}
\end{align}
Now, when $\tilde T_i'' \geq T''$, the difference between $\tilde T_i''$ and $T''$ can be bounded by considering the case where all points from $\mathcal{B}_i$ are removed.  Let $T_{\mathcal{B}, \mathbf{0}}''(z)$ be the $T''$-passage time from $z$ to $z + ne_1$ when paths are forbidden from using points in the box $\mathcal{B}$.  Also, let $\mathbbm{1}_{\{\mathcal{B} \mbox{\scriptsize\ used}\}}(z)$ be the indicator function of the event that the $T''$-geodesic from $z$ to $z + ne_1$ uses a point from the box $\mathcal{B}$.  Picking up with Equation \eqref{SumValue}, the right side is bounded above by:
\[
\frac{2}{|\Gamma_n|} \sum_{z \in \Gamma_n} \sum_i \mathbb{E} \left( T_{\mathcal{B}_i, \mathbf{0}}''(z) - T''(z, z + ne_1) \right)
= 2 \sum_{\mathcal{B}} \mathbb{E}\left[ (T''_{\mathcal{B},\mathbf{0}}(0) - T''(0,ne_1)) \mathbbm{1}_{\{\mathcal{B} \text{ used}\}}(0)\right].
\]
Now, most of the work for this proof is summarized in Lemma \ref{TZeroBound}, which will be used elsewhere as well, with $p = 1$.  With this, we can bound the last expression by $2Cn$.
\end{proof}

Now that we have bounded the expectation of the sum, $\sum_i |V_i|$, we turn to bounding the maximum expectation of the $|V_i|$'s.  The proof below relies on the fact that we are using $F_n$, the averaged passage time, instead of $T''$: without the averaged passage time, we could bound the expectation of each $|V_i|$ by a constant.  But, with the averaged passage time, we can use translation invariance to bound these expectations roughly by the expected length of the $T''$-geodesic as it passes through a box of size $\Gamma_n$, divided by $|\Gamma_n|$, as seen in Equation \eqref{PartialAvgBoxBound} below.  Ultimately, we obtain a bound of the order of the diameter of $\Gamma_n$, divided by its volume, as in Equation \eqref{EVFinalBound}. In the proof (and elsewhere in the paper) we make the following distinction between ``using'' and ``touching'' a box. A path $r=(r_0, r_1, \dots, r_k)$ uses a box if some $r_i$ is inside the box, whereas it touches the box if any of its line segments intersect the box. In these definitions, $0$ and $ne_1$ are always used by the $T''$-geodesic from $0$ to $ne_1$, even though they are almost-surely not Poisson points.

\begin{lem} \label{MaxV}
There exists a constant $C > 0$ such that for all $n \geq 1$,
\[
\sup_i \mathbb{E}|V_i| \leq Cn^{-1/4\alpha}.
\]
\end{lem}

\begin{proof}

As in the proof of Lemma \ref{lin}, we let $\tilde T_i''$ represent $T''$ after the points in box $\mathcal{B}_i$ are resampled.  First, we note that, as in \eqref{eq: v_i_equiv}:
\begin{equation}\label{eq: new_v_i_equiv}
\mathbb{E}|V_i| \leq \frac{1}{|\Gamma_n|} \sum_{z \in \Gamma_n} \mathbb{E}|T''(z, z + ne_1) - \tilde T_i''(z, z + ne_1)|.
\end{equation}
Let $\mathcal{B}_i + \Gamma_n = \{\mathcal{B}_i + z : z \in \Gamma_n\}$ represent the set of boxes $\mathcal{B}$ that are translates of $\mathcal{B}_i$ by some element $z \in \Gamma_n$, and let $\mathbbm{1}_{\{\mathcal{B} \mbox{\scriptsize\ used}\}}(0)$ be the indicator function of the event that the $T''$-geodesic from $0$ to $ne_1$ uses a point from the box $\mathcal{B}$.

For any box $\mathcal{B}_m$ that is used by a $T''$-geodesic, let $s_m^-$ be the first point the $T''$-geodesic uses from $\mathcal{B}_m$, and let $r_m^-$ be the point immediately preceding $s_m^-$.  Similarly, let $s_m^+$ be the last point the $T''$-geodesic uses from $\mathcal{B}_m$, and let $r_m^+$ be the next point the $T''$-geodesic uses after $s_m^+$.  Finally, if the box $\mathcal{B}_m$ contains $0$, let $r_m^- = s_m^- = 0$.  Likewise, if $\mathcal{B}_m$ contains $ne_1$, let $r_m^+ = s_m^+ = ne_1$.

By translation invariance, we can reindex the sum in \eqref{eq: new_v_i_equiv} to obtain the following:

\begin{align*}
\mathbb{E}|V_i|
&\leq \frac{1}{|\Gamma_n|} \sum_{\mathcal{B}_m \in \mathcal{B}_i + \Gamma_n} \mathbb{E} \big| T''(0, ne_1) - \tilde T_m''(0, ne_1) \big|\\
&\leq \frac{2}{|\Gamma_n|} \sum_{\mathcal{B}_m \in \mathcal{B}_i + \Gamma_n} \mathbb{E}\big| T''(0,ne_1) - \tilde T_m''(0,ne_1)\big| \mathbbm{1}_{\{\tilde T_m'' \geq T''\}} \\
&\leq \frac{2}{|\Gamma_n|} \sum_{\mathcal{B}_m \in \mathcal{B}_i + \Gamma_n} \mathbb{E} \big| \phi_n(\lVert r_m^- - r_m^+\rVert) \mathbbm{1}_{\{\mathcal{B}_m \mbox{\scriptsize\ used}\}}(0) \big|,
\end{align*}
which is very similar to \eqref{UsedPhiIneq} from the proof of Lemma \ref{TZeroBound}.  Then, we can use the modified triangle inequality for $\phi_n$ described in Lemma \ref{PhiIneq} to bound this by the following:
\begin{equation} \label{PreviousPartialAvgBoxBound}
\frac{2^{2\alpha +1}}{|\Gamma_n|}\, \mathbb{E}\left[\sum_{\mathcal{B}_m \in \mathcal{B}_i + \Gamma_n} \big[ \phi_n(\lVert r_m^- - s_m^-\rVert) + \phi_n(\lVert s_m^- - s_m^+\rVert) + \phi_n(\lVert s_m^+ - r_m^+\rVert)  \big] \mathbbm{1}_{\{\mathcal{B}_m \mbox{\scriptsize\ used}\}}(0) \right].
\end{equation}
Let $\mathcal{B}_{t_1}$ be the first box in $\mathcal{B}_i + \Gamma_n$ the $T''$-geodesic from $0$ to $ne_1$ uses, and let $\mathcal{B}_{t_2}$ be the last box it uses before leaving $\mathcal{B}_i + \Gamma_n$ for the last time if it does indeed leave, or the box containing $ne_1$ if it does not.  Let $\mathcal{C}(x, y)$ be the set of boxes the $T''$-geodesic from $x$ to $y$ touches, and let $\#\mathcal{C}(x, y)$ be the number of these boxes.  Because $s_{t_1}^-$ and $s_{t_2}^+$ are already points in $Q_n$ (or $0$ or $ne_1$), $\#\mathcal{C}(s_{t_1}^-, s_{t_2}^+)$ is the number of boxes the $T''$-geodesic from $0$ to $ne_1$ touches while using boxes in $\mathcal{B}_i + \Gamma_n$.

Note that each $\phi_n(\lVert r_m^- - s_m^- \rVert)$ and $\phi_n(\lVert s_m^+ - r_m^+\rVert)$ in Equation \eqref{PreviousPartialAvgBoxBound} corresponds to the length of a segment in the original $T''$-geodesic.  Recognizing that it is possible for $r_k^- = s_j^+$ for some pair $j \neq k$, we can bound the sum of these terms by $2T''(s_{t_1}^-, s_{t_2}^+) + \phi_n(\lVert r_{t_1}^- - s_{t_1}^-\rVert) + \phi_n(\lVert s_{t_2}^+ - r_{t_2}^+\rVert)$, where we separated the segments entering and leaving $\mathcal{B}_i + \Gamma_n$.  Additionally, each $\phi_n(\lVert s_m^- - s_m^+\rVert)$ is bounded deterministically by a constant, because the $s_m^{\pm}$ are within the same unit box.  Combining these, we can use Equation \eqref{PreviousPartialAvgBoxBound} to obtain the following inequality for some $C > 0$:
\begin{equation} \label{PartialAvgBoxBound}
\mathbb{E} |V_i| \leq \frac{C}{|\Gamma_n|} \bigg[ \mathbb{E} \left[\phi_n(\lVert r_{t_1}^- - s_{t_1}^-\rVert) +  T''(s_{t_1}^-, s_{t_2}^+) + \phi_n(\lVert s_{t_2}^+ - r_{t_2}^+\rVert) \right] + \mathbb{E} \#\mathcal{C}(s_{t_1}^-, s_{t_2}^+) \bigg].
\end{equation}
Lemma \ref{MaxLengthBound} below analyzes the maximum length of segments in a $T''$-geodesic.  From it, we can conclude that $\mathbb{E} \phi_n(\lVert r_{t_1}^- - s_{t_1}^-\rVert)$ and $\mathbb{E} \phi_n(\lVert s_m^+ - r_m^+\rVert)$ are both bounded by $C_1n^{1/4\alpha}$.

Now, in order to bound $\mathbb{E} T''(s_{t_1}^-, s_{t_2}^+)$, we will find nearly-exponential tails for $T''(s_{t_1}^-, s_{t_2}^+)$.  Bounding the tails of $T''(s_{t_1}^-, s_{t_2}^+)$ will take two steps: first, we will consider all pairs of points $(x, y)$ which are at the centers of unit boxes within $\mathcal{B}_i + \Gamma_n$, and we will show that it is unlikely for any pair to have a large $T''(x, y)$.  Then, we will relate $s_{t_1}^-$ and $s_{t_2}^+$ to the centers of the boxes containing them.  With this in mind, consider any pair of unit boxes in $\mathcal{B}_i + \Gamma_n$.  There are at most $[2 n^{1/4\alpha} + 1]^d \leq [3n^{1/4\alpha}]^d$ unit boxes in $\mathcal{B}_i + \Gamma_n$, and so there are at most $[3n^{1/4\alpha}]^{2d}$ pairs of such boxes.

Let $x$ and $y$ be the center points of any such pair of boxes.  Then, $\lVert x - y \rVert \leq 2\sqrt{d}n^{1/4\alpha}$, and so by Lemma \ref{TprimeTails}, we have for $C_1$ sufficiently large and $c_2$ sufficiently small,
\[
\mathbb{P}(T''(x, y) \geq \ell) \leq C_1\exp\left(-c_2 \ell^\kappa\right)
\]
for $\kappa = \min(1, d/\alpha)$, assuming $\ell \geq C_1n^{1/4\alpha}$.   Then, let $z_{\mathcal{B}_i}$ be the center of box $\mathcal{B}_i$, and let $\mathcal{D} = \{x \in \mathbb{Z}^d: \lVert x - z_{\mathcal{B}_i} \rVert_\infty \leq n^{1/4\alpha}\}$ be the set of points that are the centers of boxes in $\mathcal{B}_i + \Gamma_n$.  We have:
\begin{align}
\mathbb{P}\left(\max_{x, y \in \mathcal{D}} T''(x, y) \geq \ell \right) \leq \sum_{x, y \in \mathcal{D}} \mathbb{P}\left(T''(x, y) \geq \ell\right) &\leq \left[3n^{1/4\alpha}\right]^{2d} C_1\exp(-c_2 \ell^\kappa) \nonumber \\
&\leq C_1\exp\left(-c_2 \ell^\kappa\right), \label{TprimeAverageCenterBound}
\end{align}
where on the last line, $C_1$ is larger and $c_2$ is smaller.

Now, let $a$ and $b$ be the center points of the unit boxes containing $s_{t_1}^-$ and $s_{t_2}^+$, respectively.  Then, let $(r_0, r_1, \ldots, r_{k + 1})$ be the $T''$-geodesic connecting $a$ and $b$, with $a = r_0$ and $r_{k + 1} = b$.  We break into cases depending on whether $k > 0$ or $k = 0$.  Assuming $k > 0$, one possible path between $s_{t_1}^-$ and $s_{t_2}^+$ is given by $s_{t_1}^-, r_1, r_2, \ldots, r_k, s_{t_2}^+$.  Note that this path cannot use points $a$ and $b$, because they are almost surely not Poisson points in $Q$.  Then,
\begin{align}
T''\left(s_{t_1}^-, s_{t_2}^+\right) &\leq \phi_n \left( \lVert s_{t_1}^- - r_1 \rVert \right) + \phi_n \left( \lVert s_{t_2}^+ - r_k \rVert \right) + \bigg[ T''(a, b) - \phi_n \left( \lVert r_1 - a \rVert \right) - \phi_n \left( \lVert r_k - b \rVert \right) \bigg] \nonumber \\
&\leq 2^\alpha \bigg[ \phi_n \left( \lVert s_{t_1}^- - a \rVert \right) + \phi_n \left( \lVert a - r_1 \rVert \right) + \phi_n \left( \lVert s_{t_2}^+ - b \rVert \right) + \phi_n \left( \lVert b - r_k \rVert \right) \bigg] \nonumber \\
& \ \ \ \ + \bigg[ T''(a, b) - \phi_n \left( \lVert r_1 - a \rVert \right) - \phi_n \left( \lVert r_k - b \rVert \right) \bigg] \nonumber\\
&\leq 2^\alpha \bigg[ \phi_n \left( \lVert a - s_{t_1}^- \rVert \right) + \phi_n \left( \lVert b - s_{t_2}^+ \rVert \right) + T''(a, b) \bigg], \label{Version1}
\end{align}
where the second inequality is true by using the modified triangle inequality for $\phi_n$ given by Lemma \ref{PhiIneq}.  Now, for the case when $k = 0$, the $T''$-geodesic between $a$ and $b$ is the direct path from $a$ to $b$.  By using Lemma \ref{PhiIneq} twice, we have:
\begin{align}
T''\left(s_{t_1}^-, s_{t_2}^+\right) \leq \phi_n \left( \lVert s_{t_1}^- - s_{t_2}^+ \rVert \right) &\leq 2^{2\alpha}\left[ \phi_n \left( \left \lVert s_{t_1}^- - a \right \rVert \right) + \phi_n\left( \left\lVert a - b \right \rVert \right) + \phi_n \left( \lVert b - s_{t_2}^+ \rVert \right)\right] \nonumber\\
&= 2^{2\alpha}\left[ \phi_n \left( \left \lVert s_{t_1}^- - a \right \rVert \right) + \phi_n \left( \lVert b - s_{t_2}^+ \rVert \right)+ T''(a, b)\right]. \label{Version2}
\end{align}
We see that between Equations \eqref{Version1} and \eqref{Version2}, the second one is weaker and thus holds for any value of $k \geq 0$.  Since $a$ and $s_{t_1}^-$ are in the same $d$-dimensional unit box, $\phi_n\left(\lVert a - s_{t_1}^- \rVert\right) \leq \phi_n( \sqrt{d}) \leq d^{\alpha/2}$.  The same is true for $s_{t_2}^+$ and $b$.  Thus, rearranging Equation \eqref{Version2} gives:
\[
T''(a, b) \geq 2^{-2\alpha} T''(s_{t_1}^-, s_{t_2}^+) - 2d^{\alpha/2}.
\]
So, if $T''(s_{t_1}^-, s_{t_2}^+) \geq 2^{2\alpha+1}\ell$ for $\ell \geq C_1 n^{1/4\alpha}$ and $C_1$ is sufficiently large, then $T''(a, b) \geq  \ell$.  Thus,
\begin{align*}
 \mathbb{P}\left( T''(s_{t_1}^-, s_{t_2}^+) > 2^{2\alpha+1} \ell \right) \leq  \mathbb{P}\left( T''(a, b) > \ell \right) &\leq \mathbb{P}\left(\max_{(x, y) \in \mathcal{D}} T''(x, y) \geq \ell \right) \nonumber\\
 &\leq C_1 \exp\left(-c_2 \ell^\kappa\right), \label{TPrimeAvgPart1}
\end{align*}
from Equation \eqref{TprimeAverageCenterBound}, for perhaps a smaller $c_2$ and larger $C_1$, where $\kappa = \min(1, d/\alpha)$.  From this, we can conclude that for $C$ sufficiently large,
\[
\mathbb{E} T''(s_{t_1}^-, s_{t_2}^+) \leq C_1n^{1/4\alpha} + \int_{C_1n^{1/4\alpha}}^\infty \mathbb{P}(T''(s_{t_1}^-,s_{t_2}^+) \geq \lambda)~\text{d}\lambda \leq C n^{1/4\alpha}.
\]
Thus, it remains to bound $\mathbb{E}\#\mathcal{C}(s_{t_1}^-, s_{t_2}^+)$ in Equation \eqref{PartialAvgBoxBound}.  To do so, consider all the unit boxes $\mathcal{B} \in \mathcal{B}_i + \Gamma_n$.  For any two boxes $\mathcal{B}_j, \mathcal{B}_k \in \mathcal{B}_i + \Gamma_n$, the centers of the boxes are within $2\sqrt{d} n^{1/4\alpha}$ of each other.  By choosing $\gamma = 1/4\alpha$ and $C_3 = 2\sqrt{d}$ in Lemma \ref{BoxBound}, we can conclude that for $C_1$ large enough and some $c_2$, for all $\ell \geq C_1 n^{1/4\alpha}$,
\begin{equation} \label{Gamma-BoxBound}
\mathbb{P}\left( \max_{v \in B_j, w \in B_k} \# \mathcal{C}(v, w) \geq \ell\right) \leq C_1 \exp\left(-c_2 \ell^\kappa\right),
\end{equation}
where $\kappa = \min(1, d/\alpha)$ again.  Therefore,
\begin{align*}
\mathbb{P}\left( \# \mathcal{C}(s_{t_1}^-, s_{t_2}^+) \geq \ell \right) 
&\leq \sum_{\mathcal{B}_j, \mathcal{B}_k \in \mathcal{B}_i + \Gamma_n} \mathbb{P}\left( \max_{w \in \mathcal{B}_j, v \in \mathcal{B}_k} \#\mathcal{C}(w, v) \geq \ell\right)\\
&\leq C_4 \left(n^{1/4\alpha}\right)^{2d} C_1 e^{-c_2 \ell^{\kappa}},
\end{align*}
where the last line is true for some $C_4 > 0$ by using Equation \eqref{Gamma-BoxBound} and also bounding the total number of pairs of unit boxes within $\Gamma_n$.  In turn, this implies:
\[
\mathbb{P}\left( \# \mathcal{C}(s_{t_1}^-, s_{t_2}^+) \geq \ell \right) \leq C_1 \exp\left(-c_2\ell^\kappa\right)
\]
for a potentially larger choice of $C_1$ and smaller choice of $c_2$, since $\ell \geq C_1 n^{1/4\alpha}$.  These exponential tails imply that $\mathbb{E} \# \mathcal{C}(s_{t_1}^-, s_{t_2}^+) \leq C_1 n^{1/4\alpha}$ for a potentially larger $C_1$.  Plugging this into Equation \eqref{PartialAvgBoxBound} and recalling that $d \geq 2$ gives for some large enough $C_1$ and for each $i$,
\begin{equation} \label{EVFinalBound}
\mathbb{E}|V_i| \leq C_1 \frac{n^{1/4\alpha}}{|\Gamma_n|} \leq C_1 \frac{n^{1/4\alpha}}{\left(n^{1/4\alpha} \right)^2} = C_1n^{-1/4\alpha}.
\end{equation}
This completes the proof.
\end{proof}
\hrule
\section{Bounding entropy} \label{Entropy}
 
 The goal of this section will be to derive the following bound:
 \begin{lem} \label{EntropyBoundGoal}
 There exists a constant, $C$, such that for all $n \geq 1$,
 \[
 \sum_i \Ent V_i^2 \leq Cn.
 \]
 \end{lem}
Recall that we decompose the Poisson process which determines the points in $Q_n$ into independent processes on each unit box $\mathcal{B}$ in $\mathbb{R}^d$.  Within each unit box $\mathcal{B}$, a Poisson random variable determines how many points are in the box, and uniform random variables determine the positions of the points.  Following the techniques in \cite{Damron:2015}, we will encode the Poisson random variable through an infinite sequence of Bernoulli random variables, described below.  Once this is done, we can apply logarithmic Sobolev inequalities to argue that the sum of the entropy, $\sum_i \Ent V_i^2$, is bounded by the sum of derivatives with respect to these uniform and Bernoulli random variables, as described in Lemma \ref{EntropyDerivativeBound} below.

To begin, define $\ol{\Omega} = \prod_\mathcal{B} \Omega_\mathcal{B}$, where for each unit box $\mathcal{B}$, $\Omega_\mathcal{B} = \Omega_{\mathcal{B}, 1} \times \Omega_{\mathcal{B}, 2}$ with $\Omega_{\mathcal{B}, 1} = \{0, 1\}^{\mathbb{N}}$ and $\Omega_{\mathcal{B}, 2} = ([0, 1]^d)^{\mathbb{N}}$.  Then, define the measure on $\Omega_\mathcal{B}$ to be $\pi_\mathcal{B}$, a product measure $\left(\prod_{j \geq 1} \pi_{\mathcal{B}, j}\right) \times \left(\prod_{j \geq 1} \nu_{\mathcal{B}, j}\right)$, where $\pi_{\mathcal{B}, j}$ is uniform on $\{0, 1\}$, and $\nu_{\mathcal{B}, j}$ is uniform on $[0, 1]^d$.  Also, the measure on $\ol{\Omega}$ will be $\pi = \prod_{\mathcal{B}} \pi_\mathcal{B}$. Here, the $\pi_{\mathcal{B}, j}$ correspond to the Bernoulli encodings of the Poisson random variable corresponding to the number of points in box $\mathcal{B}$.  The $\nu_{\mathcal{B}, j}$ correspond to the location of the $j$th Poisson point (if it exists) within the box.

The binary sequence within $\Omega_{\mathcal{B}, 1}$ will be converted to a uniform number between $0$ and $1$, which can be used with the inverse of the cumulative distribution function of the Poisson distribution to define a Poisson random variable.  More explicitly, let $\mathcal{D}(x)$ be the cumulative distribution function for the Poisson distribution and write $\mathcal{D}^{-1}$ for its generalized inverse,
\[
\mathcal{D}^{-1}(t) = \inf\{x : \mathcal{D}(x) \geq t \}.
\]
We will represent elements of $\Omega_{\mathcal{B}, 1}$ as $\omega = (\omega_1, \omega_2, \ldots)$, and elements of $\Omega_{\mathcal{B}, 2}$ as $U = (u_1, u_2, \ldots)$.  Let $\Psi_\mathcal{B}: \Omega_\mathcal{B} \to \mathbb{N} \times \Omega_{\mathcal{B}, 2}$ be defined by:
\[
\Psi_\mathcal{B}(\omega, U) = \Psi_\mathcal{B}(\omega_1, \omega_2, \ldots, u_1, u_2, \ldots) := \left( \mathcal{D}^{-1} \left(\sum_{i = 1}^\infty \omega_i/2^i\right), u_1, u_2, \ldots \right).
\]
$\Psi_\mathcal{B}$ is a measurable function for each $\mathcal{B}$, and thus so is $\Psi := \prod_{\mathcal{B}} \Psi_\mathcal{B} : \ol{\Omega} \to \Omega$, which pushes the corresponding measure $\pi$ forward to our original Poisson process measure, $\mathbb{P}$, on $\Omega$.

Now, we can set $\ol{F_n} := F_n \circ \Psi$, and in this way, we can view $F_n$ as a function on $\ol{\Omega}$.   We will relate entropy to derivatives of $\ol{F_n}$ with respect to the Bernoulli and uniform random variables in the input of $\ol{F_n}$.  To do so, write any element of $\ol{\Omega}$ as 
\[
(\omega_{\mathcal{B}},U_{\mathcal{B}})_{\mathcal{B}} = (\omega_{\mathcal{B}, 1},\omega_{\mathcal{B},2}, \ldots, U_{\mathcal{B}, 1}, U_{\mathcal{B}, 2}, \ldots)_{{\mathcal{B}}},
\]
where each pair $(\omega_{\mathcal{B}}, U_{\mathcal{B}}) \in \Omega_{\mathcal{B}}$ represents the Bernoulli encoding of the Poisson random variable and the uniform random variables from box $\mathcal{B}$.  Then, define for any function $f : \ol{\Omega} \to \mathbb{R}$ the discrete derivatives,
\[
\left(\Delta_{\omega_\mathcal{B}, j} f\right)(\omega, U) = f\left(\omega_{\mathcal{B}, j}^+\right) - f\left(\omega_{\mathcal{B},j}^-\right),
\]
where $\omega_{\mathcal{B}, j}^+$ represents that the element $\omega_{\mathcal{B}, j}$ in $(\omega, U)$ is replaced by a $1$, and $\omega_{\mathcal{B}, j}^-$ represents that $\omega_{\mathcal{B}, j}$ has been replaced by a $0$.  Additionally, define derivatives with respect to the uniform random variables,
\[
\lVert \nabla_{U_{\mathcal{B}, i}} \ol{F_n}(\omega, U) \rVert^2 := \sum_{j = 1}^d \left( \frac{\partial}{\partial U_{\mathcal{B}, i, j}} \ol{F_n}(\omega, U)\right)^2.
\]
Here, the index, $i$, corresponds to the $i$th uniform random variable in box $\mathcal{B}$, where the index, $j$, corresponds to the $j$th component of this random variable in $[0, 1]^d$.  With this notation, we are able to state our result about a bound on the sum of the entropies:
\begin{lem} \label{EntropyDerivativeBound}
Let $\ol{F_n}$ be the pushforward of $F_n$ onto the space $\bar \Omega$, as defined above.  Then, there exists a constant $C > 0$ such that:
\[
\sum_i \Ent V_i^2 \leq C \sum_\mathcal{B} \left[ \sum_j \mathbb{E}_\pi \left(\Delta_{j, \mathcal{B}} \ol{F_n} \right)^2 + \sum_k \mathbb{E}_\pi \left\lVert \nabla_{U_{\mathcal{B}, k}} \ol{F_n} \right\rVert^2\right].
\]
\end{lem}

\begin{proof}
We follow the proof of Lemma~6.3 in \cite{Damron:2015}.  We will use tensorization of entropy to divide the entropy into the sums of entropies over the measures defined on each $\Omega_\mathcal{B}$.  Then, we can use logarithmic Sobolev inequalities to convert entropies into derivatives.

First, we define a filtration of $\ol{\Omega}$ by enumerating the unit boxes in $\mathbb{R}^d$ as $\mathcal{B}_1, \mathcal{B}_2, \ldots$, and defining $\ol{\mathcal{F}_i} = \sigma\big(\{\omega_{\mathcal{B}_j} : j \leq i\} \cup \{U_{\mathcal{B}_j} : j \leq i\} \big)$ to be the sigma-algebra generated by the variables associated to boxes $\mathcal{B}_1$ through $\mathcal{B}_i$.  Then, we define $V_i = \mathbb{E}_\pi[\ol{F_n} | \ol{\mathcal{F}_i}] - \mathbb{E}_\pi[\ol{F_n} | \ol{\mathcal{F}_{i - 1}}].$ One can check that $\mathbb{E}[F_n | \mathcal{F}_i](\Psi(\omega, U)) = \mathbb{E}[\ol{F_n} | \ol{\mathcal{F}_i}](\omega, U)$ for $\pi$-almost every $(\omega, U) \in \ol{\Omega}$ (and similarly for $i-1$).  As a result, $\Ent_\pi \ol{V_i}^2 = \Ent V_i^2$ for each $i$, where $\Ent_\pi$ is defined in terms of the measure $\pi$ on $\ol{\Omega}$.

Tensorization of entropy (the version we use is Theorem 2.3 from \cite{Damron:2015}) will allow us to break up the $\Ent_\pi \ol{V_i}^2$ terms into sums over the component measures on the $\Omega_{\mathcal{B}}$, but we must first show that $\ol{V_i}^2 \in L^2$ for all $i$.  Equivalently, we need to show that $V_i \in L^4$, and it will suffice to show that $F_n \in L^4$.   Here, by Jensen's inequality,
\[
\mathbb{E} F_n^4 = \mathbb{E} \left(\frac{1}{|\Gamma_n|}\sum_{z \in \Gamma_n} T''(z, z + ne_1)\right)^4 \leq \frac{1}{|\Gamma_n|} \mathbb{E} \sum_{z \in \Gamma_n} T''(z,z+ne_1)^4  = \mathbb{E}T''(0,ne_1)^4.
\]
From Lemma 3.3 of \cite{Howard:2001} reproduced as Lemma \ref{TTails} below, $T''(z, z + ne_1)$ has nearly-exponential tails, and thus all moments of $T''(z, z + ne_1)$ exist.  So, using Theorem 2.3 of \cite{Damron:2015} gives us:
\begin{equation*}\label{PartialEntBound}
\sum_i \Ent V_i^2 = \sum_i \Ent_\pi \ol{V_i}^2 \leq \sum_{i = 1}^\infty \mathbb{E}_\pi \sum_{\mathcal{B}} \left[ \sum_{j = 1}^\infty \Ent_{\pi_{\mathcal{B}, j}} \ol{V_i}^2 + \sum_{k = 1}^\infty \Ent_{\nu_{\mathcal{B}, k}} \ol{V_i}^2\right].
\end{equation*}
Now, we apply two logarithmic Sobolev inequalities.  First, we use the Bonami-Gross inequality from \cite{Bonami:1970} and \cite{Gross:1975}.  Our version is stated as Theorem \ref{logSobolevBernoulli} below, taken from Theorem 5.1 in \cite{Boucheron:2013}.  It tells us that the entropy over a Bernoulli distribution is bounded by the discrete derivative, so that:
\begin{equation} \label{BernoulliVBound}
\sum_{i = 1}^\infty \mathbb{E}_\pi \sum_{\mathcal{B}} \sum_{j = 1}^\infty \Ent_{\pi_{\mathcal{B}, j}} \ol{V_k}^2 \leq \sum_{i = 1}^\infty \sum_{\mathcal{B}} \sum_{j = 1}^\infty \mathbb{E}_\pi (\Delta_{\omega_\mathcal{B}, j} \ol{V_i})^2.
\end{equation}
For the entropies over the uniform measures, we have the theorem from section 8.14 of \cite{Lieb:2001}, written as Theorem \ref{logSobolevUniform} below.  We will show that $\mathbb{E}_\pi \lVert \nabla_{U_{\mathcal{B}, i}} \ol{V_i} \rVert^2 < \infty$ in Lemma \ref{UniformBound} below.  In order to apply the uniform log-Sobolev inequality, we can take $f = \ol{V_i}(\omega, U)$ as a function of $U_{\mathcal{B}, k}$ for some $\mathcal{B}$ and $k$, where all other elements of $\omega$ and $U$ are fixed.  Also, we can extend the definition of this $f$ so that $f$ is zero when the input $U_{\mathcal{B}, k}$ is not in $[0, 1]^d$.  Then, we can plug this $f$ into Theorem \ref{logSobolevUniform} and choose $a = \sqrt{\pi}$ to obtain:
\begin{equation} \label{UniformVBound}
\sum_{i = 1}^\infty \mathbb{E}_\pi \sum_{\mathcal{B}} \sum_{k = 1}^\infty \Ent_{\nu_{\mathcal{B}, k}} \ol{V_i}^2 \leq \sum_{i = 1}^\infty \sum_\mathcal{B} \sum_{k = 1}^\infty \mathbb{E}_\pi \lVert \nabla_{U_{\mathcal{B}, k}} \ol{V_i} \rVert^2.
\end{equation}
Now, to finish the proof, we need to rewrite the bounds in Equations \eqref{BernoulliVBound} and \eqref{UniformVBound} in terms of $\ol{F_n}$.  First, observe that for the discrete derivatives, for any box $\mathcal{B}_m$, and any $j \geq 1$,
\[
\Delta_{\omega_{\mathcal{B}_m, j}} \ol{V_i} =
\left\{
\begin{array}{ll}
\mathbb{E}_\pi[\Delta_{\omega_{\mathcal{B}_m, j}}\ol{F_n} | \ol{\mathcal{F}_i}] - \mathbb{E}_\pi [\Delta_{\omega_{\mathcal{B}_m, j}}\ol{F_n} | \ol{\mathcal{F}_{i - 1}}] & i \geq m,\\
\mathbb{E}_\pi[\Delta_{\omega_{\mathcal{B}_m, j}}\ol{F_n} | \ol{\mathcal{F}_i}] & i = m,\\
0 & i < m.
\end{array}
\right.
\]
The reason is that $\mathbb{E}_\pi[\ol{F_n} | \ol{\mathcal{F}_i}] $ is a function of the variables associated to boxes $\mathcal{B}_1$ through $\mathcal{B}_i$, and hence has zero derivative with respect to variables from boxes after $\mathcal{B}_i$.  Then, by the orthogonality of martingale differences, $\sum_{i = 1}^\infty \mathbb{E}_\pi \left( \Delta_{\omega_\mathcal{B}, j} \ol{V_i}\right)^2 = \mathbb{E}_\pi \left(\Delta_{\omega_\mathcal{B}, j} \ol{F_n} \right)^2$.

Because $\lVert \nabla_{U_\mathcal{B}, k} \ol{V_i}\rVert^2$ can be expanded as the sum of squares of partial derivatives with respect to the $d$ coordinates of $U_\mathcal{B}$, we can treat the partial derivatives in each of the $d$ coordinates individually as we did with the discrete derivatives and obtain $\sum_{i = 1}^\infty \mathbb{E}_\pi \lVert \nabla_{U_\mathcal{B}, k} \ol{V_i}\rVert^2 = \mathbb{E}_\pi \lVert \nabla_{U_\mathcal{B}, k} \ol{F_n}\rVert^2.$  Plugging these equalities into Equations \eqref{BernoulliVBound} and \eqref{UniformVBound} completes the proof.
\end{proof}

Let $\ol{T}''$ represent the pushforward of $T''$ onto the space, $\ol{\Omega}$.  Note that we can bound the derivatives of $\ol{F_n}(0, ne_1)$ by the derivatives of $\ol{T}''(0, ne_1)$, because the derivative operator is linear.  For example, for the discrete derivatives,
\begin{align*}
\mathbb{E}_\pi(\Delta \ol{F_n})^2 = \mathbb{E}_\pi \left[\Delta \left(\frac{1}{|\Gamma_n|} \sum_{z \in \Gamma_n} \ol{T}''(z, z + ne_1)\right)\right]^{2} &\leq \frac{1}{|\Gamma_n|} \sum_{z \in \Gamma_n} \mathbb{E}_\pi \left(\Delta \ol{T}''(z, z + ne_1)\right)^{2}\\
&= \mathbb{E}_\pi \left( \Delta \ol{T}''(0, ne_1) \right)^2,
\end{align*}
where the last line is true by translation invariance.  Because $\left\lVert \nabla_{U_{\mathcal{B}, k}} \ol{F_n} \right\rVert^2$ is the sum of squares of partial derivatives, the same line of arguments holds to show $\mathbb{E}_\pi \left\lVert \nabla_{U_{\mathcal{B}, k}} \ol{F_n} \right\rVert^2 \leq \mathbb{E}_\pi \left\lVert \nabla_{U_{\mathcal{B}, k}} \ol{T}'' \right\rVert^2$.  Therefore, in order to complete the proof of Lemma \ref{EntropyBoundGoal}, we need to bound the derivatives of $\ol{T}''$, as outlined in Lemmas \ref{UniformBound} and \ref{BernoulliBound} below.

\begin{lem}\label{UniformBound} There exists a constant $C > 0$ such that
\[
\sum_\mathcal{B} \sum_k \mathbb{E}_\pi \left\lVert \nabla_{U_{\mathcal{B}, k}} \ol{T}''(0, ne_1)\right\rVert^2 < Cn.
\]
\end{lem}

\begin{proof}
Throughout, we let $\ol{T}''$ represent $\ol{T}''(0, ne_1)$.  Let $p_\mathcal{B}$ be the number of Poisson points in box $\mathcal{B}$. Then, 
\begin{equation*}\label{eq: pasta}
\mathbb{E}_\pi \left\lVert \nabla_{U_{\mathcal{B}, k}} \ol{T}'' \right\rVert^2 = \mathbb{E}_\pi \left[ \left\lVert \nabla_{U_{\mathcal{B}, k}} \ol{T}'' \right\rVert^2 \mathbbm{1}_{\{p_\mathcal{B} \geq k\}}\right].
\end{equation*}
We consider how much a distance can change as the $k$th uniform random variable, $U_{\mathcal{B},k}$ changes.  Consider a portion of three consecutive points $(w_0, w_1, w_2)$ from the $T''$-geodesic from $0$ to $ne_1$.  We will analyze the change in $\ol{T}''$ as the random variable controlling the placement of $w_1$ is varied.  Because we care only about distances between the points, we will assume $w_0 = 0$ without loss of generality, making the points on the $T''$-geodesic $(0, w_1, w_2)$.  Restricting perturbations of $w_1$ temporarily to only the direction, $e_1$, consider the difference in $\ol{T}''$ as $w_1$ is perturbed by $te_1$ for a small positive or negative number, $t$:
\begin{equation} \label{GeoDistance}
\phi_n(\lVert w_1 \rVert) + \phi_n(\lVert w_1 - w_2 \rVert) - \phi_n(\lVert w_1 + te_1 \rVert) - \phi_n(\lVert w_1 + te_1 - w_2 \rVert)
\end{equation}
The expression above has two terms of the form, $\phi_n(\lVert y \rVert) - \phi_n(\lVert y + te_1 \rVert)$, for $y = w_1$ or $y = w_1 - w_2$, and we can analyze both in the same framework.  Based on the definition of $\phi_n$, we will need to break into cases depending on whether $\lVert y \rVert < h_n$ or $\lVert y \rVert > h_n$.  Note that since the $w_i$'s are Poisson points, we may assume $\|w_1\| \neq h_n$ and $\|w_2-w_1\|\neq h_n$ in our calculations. (This is also valid if the entire geodesic has only three points.)

If $t$ is sufficiently small, both $\lVert y \rVert$ and $\lVert y + te_1\rVert$ are above $h_n$, or both are below $h_n$.  For the first case, assume that $\lVert y \rVert$ and $\lVert y + te_1\rVert \leq h_n$, and write $\lVert y \rVert = \ell$.  Because $\phi_n$ is an increasing function in $t$, this difference between $\phi_n(\lVert y \rVert)$ and $\phi_n(\lVert y + te_1\rVert)$ is maximized when $0, y,$ and $te_1$ are colinear.  So, as $t$ approaches zero,
\begin{align*}
\left| \frac{\phi_n(\lVert y \rVert) - \phi_n(\lVert y + te_1\rVert)}{t} \right| = \left| \frac{\lVert y \rVert^\alpha - \lVert y + te_1 \rVert^\alpha}{t} \right| = \left| \frac{\ell^\alpha - (\ell + t)^\alpha}{t} \right| & \leq \frac{\alpha t \ell^{\alpha - 1} + O(t^2)}{t}\\
 & = \alpha \ell^{\alpha - 1} + O(t).
\end{align*}
In the other case, assume that $\lVert y \rVert$ and $\lVert y + te_1 \rVert > h_n$.  Then,
\begin{align*}
\left| \frac{\phi_n(\lVert y \rVert) - \phi_n(\lVert y + te_1\rVert)}{t} \right| &= \left| \frac{h_n^\alpha + \alpha h_n^{\alpha - 1}(\ell - h_n) - h_n^\alpha - \alpha h_n^{\alpha - 1}(\ell + t - h_n)}{t} \right|\\
&= \alpha h_n^{\alpha - 1} \leq \alpha \ell^{\alpha - 1}.
\end{align*}
Therefore, in both cases, the expression in \eqref{GeoDistance} is bounded by the following as $t$ approaches zero:
\begin{equation*} \label{DistanceBound}
\alpha t \cdot \left[\lVert w_1 \rVert^{\alpha - 1} + \lVert w_1 - w_2 \rVert^{\alpha - 1}\right] + O(t^2).
\end{equation*}

To use this to bound the derivative, let $t \in \mathbb{R}$ and, in a given Poisson configuration and given $\mathcal{B}$ and $i,k$, write $T''$ for the passage time from $0$ to $ne_1$ in this configuration, and $T''(t)$ for the passage time from $0$ to $ne_1$ in the configuration in which the uniform variable $U_{\mathcal{B},k,i}$ is replaced by $U_{\mathcal{B},k,i}+t$. Then one has
\[
T''(t) - T'' \leq T''(\mathcal{G}(t)) - T''(\mathcal{G}),
\]
where $\mathcal{G}$ is the $T''$-geodesic from $0$ to $ne_1$ in the original configuration, and $\mathcal{G}(t)$ is this same path but with the position of the Poisson point $p$ corresponding to $U_{\mathcal{B},k}$ replaced by $U_{\mathcal{B},k}+te_i$. As in the argument of the preceding paragraph, one has an upper bound
\[
T''(t) - T'' \leq T''(\mathcal{G}(t)) - T''(\mathcal{G}) \leq \alpha t \cdot \left[ \|w_1\|^{\alpha-1} + \|w_1-w_2\|^{\alpha-1}\right] + O(t^2),
\]
where $\|w_1\|$ is the length of the segment of $\mathcal{G}$ which ends at $p$ (if it exists) and $\|w_1-w_2\|$ is the length of the segment of $\mathcal{G}$ which starts at $p$ (if it exists). By the same argument, but with $\mathcal{G}$ and $\mathcal{G}(t)$ interchanged, we obtain
\[
T'' - T''(t) \leq \alpha t \cdot \left[ \|w_1(t)\|^{\alpha-1} + \|w_1(t)-w_2(t)\|^{\alpha-1}\right] + O(t^2),
\]
where $\|w_1(t)\|$ and $\|w_1(t)-w_2(t)\|$ are the analogous quantities for the geodesic $\mathcal{G}(t)$. Because a.s. all distinct paths have distinct passage times, if $t$ is sufficiently small (depending on the original configuration), the quantities $\|w_1(t)\|$ and $\|w_1\|$ are equal (and similarly for $\|w_1(t)-w_2(t)\|$ and $\|w_1-w_2\|$). Therefore for $t$ small, we obtain
\begin{equation}\label{eq: pre_derivative}
|T'' - T''(t)| \leq \alpha t \cdot \left[ \|w_1\|^{\alpha-1} + \|w_1-w_2\|^{\alpha-1}\right] + O(t^2).
\end{equation}

Now, if the $T''$-geodesic $\mathcal{G}$ uses a point in $\mathcal{B}$, then the segments in the $T''$-geodesic with at least one point in $\mathcal{B}$ can be categorized into three types: the first segment with a point in $\mathcal{B}$, the last segment with a point in $\mathcal{B}$, and any other segment.  Let $L_{1, \mathcal{B}}$ be the Euclidean length of the first segment entering $\mathcal{B}$, and let $L_{2, \mathcal{B}}$ be the Euclidean length of the last segment. Recall that the terms $\lVert w_1 \rVert$ and $\lVert w_1 - w_2 \rVert$ in \eqref{eq: pre_derivative} correspond to the lengths of the segments attached to the point with location determined by $U_{\mathcal{B}, k}$. Thus, each are equal to $L_{1, \mathcal{B}}$ or $L_{2, \mathcal{B}}$, or they are bounded by a constant because they represent segments contained inside the unit box $\mathcal{B}$.  So, \eqref{eq: pre_derivative} implies that 
\[
|T''-T''(t)| \leq \alpha t \left[ L_{1, \mathcal{B}}^{\alpha - 1} + C + L_{2, \mathcal{B}}^{\alpha - 1}\right] + O(t^2),
\]
for some $C > 0$. Therefore, for any coordinate $i$ of the $d$ coordinates of $U_{\mathcal{B}, k}$,
\[
\left| \frac{\partial}{\partial U_{\mathcal{B}, k, i}} \ol{T}''\right| \leq C [L_{1, \mathcal{B}}^{\alpha - 1} + C + L_{2, \mathcal{B}}^{\alpha - 1}].
\]
For a potentially larger $C > 0$ (depending on the dimension, $d$) and any box, $\mathcal{B}$, this implies that the left side of the lemma is bounded by:
\begin{align}
\mathbb{E}_\pi \sum_{\mathcal{B}} \sum_k \left\lVert \nabla_{U_{\mathcal{B}, k}} \ol{T}'' \right\rVert^2 \mathbbm{1}_{\{p_\mathcal{B} \geq k\}} &\leq C \mathbb{E}_\pi \sum_{\mathcal{B}} \sum_k \left[ \left( L_{1, \mathcal{B}}^{\alpha - 1} + C + L_{2, \mathcal{B}}^{\alpha - 1} \right)^2 \mathbbm{1}_{\{p_\mathcal{B} \geq k\}} \right] \mathbbm{1}_{\{\mathcal{B}~\text{used}\}}(0) \nonumber\\
&= C \mathbb{E} \sum_{\mathcal{B}} \left[ \left( L_{1, \mathcal{B}}^{\alpha - 1} + C + L_{2, \mathcal{B}}^{\alpha - 1} \right)^2 p_\mathcal{B}\right] \mathbbm{1}_{\{\mathcal{B}~\text{used}\}}(0)\nonumber \\
& \leq 3C \mathbb{E} \sum_{\mathcal{B}} \left[ \left( L_{1, \mathcal{B}}^{2\alpha - 2} + C^2 + L_{2, \mathcal{B}}^{2\alpha - 2} \right) p_{\mathcal{B}} \right]\mathbbm{1}_{\{\mathcal{B}~\text{used}\}}(0). \label{PartialBound}
\end{align}
Here, as before, $\mathbbm{1}_{\{\mathcal{B} \mbox{\scriptsize\ used}\}}(0)$ is the indicator function of the event that the $T''$-geodesic from $0$ to $ne_1$ uses a point from the box $\mathcal{B}$.  Then, we must control the size of two types of terms:
\begin{equation} \label{PoissonBoxes}
\sum_{\mathcal{B}} \mathbb{E}\left(p_{\mathcal{B}} \mathbbm{1}_{\{\mathcal{B} \mbox{\scriptsize\ used}\}}(0)\right)
\end{equation}
and
\begin{equation} \label{LBoxes}
\sum_{\mathcal{B}} \mathbb{E}\left(L_\mathcal{B}^{2\alpha - 2} p_{\mathcal{B}} \mathbbm{1}_{\{\mathcal{B} \mbox{\scriptsize\ used}\}}(0)\right),
\end{equation}
where $L_{\mathcal{B}}$ stands for either $L_{1, \mathcal{B}}$ or $L_{2, \mathcal{B}}$.  We will show that both of these terms are at most linear in $n$.

Let us focus on \eqref{PoissonBoxes} first.  Let $\#\mathcal{C} = \#\mathcal{C}(0, ne_1)$ be the number of unit boxes the $T''$-geodesic from $0$ to $ne_1$ touches (regardless of whether it uses points from the boxes) and let $\mathcal{C}$ be the set of boxes.  We will break into cases, depending on whether $\#\mathcal{C}$ is small or large.  For $C_1$ to be restricted further later, 
\begin{equation} \label{pBoundCases}
\sum_{\mathcal{B}} \mathbb{E}\left(p_{\mathcal{B}} \mathbbm{1}_{\{\mathcal{B} \mbox{\scriptsize\ used}\}}(0) \right) = \sum_{\mathcal{B}} \mathbb{E}\left(p_{\mathcal{B}} \mathbbm{1}_{\{\mathcal{B} \mbox{\scriptsize\ used}\}}(0)\mathbbm{1}_{\{\#\mathcal{C} < C_1n\}}\right) + \sum_{\mathcal{B}} \mathbb{E}\left(p_{\mathcal{B}} \mathbbm{1}_{\{\mathcal{B} \mbox{\scriptsize\ used}\}}(0)\mathbbm{1}_{\{\#\mathcal{C} \geq C_1n\}}\right).
\end{equation}
We start with the first term.  Define a lattice animal to be any connected set of unit boxes in $\mathbb{R}^d$, where a pair of boxes is connected if it shares a face.  Let $\mathcal{A}_0(n)$ be the set of lattice animals containing at most $n$ unit boxes, and containing the unit box centered at the origin. Define:
\[
N_n := \max_{A \in \mathcal{A}_0(n)} \sum_{\mathcal{B} \in A} p_\mathcal{B}.
\]
Recalling that $\mathcal{C}$ is the set of boxes touched by the $T''$-geodesic from $0$ to $ne_1$, we see that whenever $\#\mathcal{C} < C_1n$, $\mathcal{C} \in \mathcal{A}_0(C_1n)$ almost surely. (Here we are using that almost surely, the $T''$-geodesic will only cross from one box to another through a face.) Because the Poisson distribution has moments of all orders, we can use Theorem \ref{GreedyLatticeThm} below (originally from \cite{Gandolfi:1994}) to obtain
\begin{equation} \label{pBoundPart1}
\sum_{\mathcal{B}} \mathbb{E}\left(p_{\mathcal{B}} \mathbbm{1}_{\{\mathcal{B} \mbox{\scriptsize\ used}\}}(0)\mathbbm{1}_{\{\#\mathcal{C} < C_1n\}}\right) \leq \mathbb{E} N_{C_1n}  \leq C_2n,
\end{equation}
for some positive constant $C_2$ depending on $C_1$ (which will be restricted momentarily).  For the second term in \eqref{pBoundCases}, we have the following:
\begin{align}
\sum_{\mathcal{B}} \mathbb{E}\left(p_{\mathcal{B}} \mathbbm{1}_{\{\mathcal{B} \mbox{\scriptsize\ used}\}}(0)\mathbbm{1}_{\{\#\mathcal{C} \geq C_1n\}}\right) &\leq \mathbb{E} \left[\mathbbm{1}_{\{\#\mathcal{C} \geq C_1n\}} \sum_{\mathcal{B} \in \mathcal{C}}  p_\mathcal{B}\right] \label{BeginPBound} \\
&= \sum_{k = \lceil C_1n \rceil}^\infty \mathbb{E} \left[ \mathbbm{1}_{\{\#\mathcal{C} = k\}} \sum_{\mathcal{B} \in \mathcal{C}}  p_\mathcal{B} \right] \nonumber \\
&\leq \sum_{k = \lceil C_1n \rceil}^\infty \mathbb{E} \mathbbm{1}_{\{\#\mathcal{C} \geq k\}} P_k, \nonumber
\end{align}
where $P_k$ is the number of points inside  $Q_n \cap [-k - 1/2, k + 1/2]^d$.  Then, by Cauchy-Schwarz, this is at most
\[
\sum_{k = \lceil C_1n \rceil}^\infty \sqrt{\mathbb{E} P_k^2} \sqrt{\mathbb{P}(\#\mathcal{C} \geq k)}.
\]
Because our Poisson process has rate one (but some points may be removed in $Q_n$), 
\[
\sqrt{\mathbb{E} P_k^2} \leq \sqrt{(2k + 1)^d + (2k+ 1)^{2d}} \leq 2(2k + 1)^{d}.
\]
It will be convenient to lump $n$ terms in the sum together at a time.  We see that for any integer $t$,
\begin{align*}
\sum_{k = \lceil tn \rceil}^{\lfloor (t+1)n\rfloor} \sqrt{\mathbb{E} P_k^2} \sqrt{\mathbb{P}(\#\mathcal{C} \geq k)} &\leq \sum_{k = \lceil tn \rceil}^{\lfloor (t+1)n\rfloor} \sqrt{\mathbb{E} P_{(t + 1)n}^2} \sqrt{\mathbb{P}(\#\mathcal{C} \geq tn)}\\
&\leq 2n\big[2(t + 1)n+ 1\big]^d \sqrt{\mathbb{P}(\#\mathcal{C} \geq tn)}.
\end{align*}
So, we can conclude:
\begin{align}
\sum_{k = \lceil C_1n\rceil}^\infty \sqrt{\mathbb{E} P_k^2} \sqrt{\mathbb{P}(\#\mathcal{C} \geq k)} &\leq \sum_{k = \lfloor C_1 \rfloor}^\infty 2n\big[2(k + 1)n+ 1\big]^d \sqrt{\mathbb{P}(\# \mathcal{C} \geq kn)} \nonumber \\
&\leq \sum_{k = \lfloor C_1 \rfloor}^\infty 2n\big[2(k + 1)n + 1\big]^d C_3 e^{-c_4(kn)^\kappa} \nonumber \\
&\leq C_5 e^{-c_6n^\kappa}, \label{LongPoissonBound}
\end{align}
where $C_3, c_4, C_5,$ and $c_6$ are positive constants, and the second inequality is due to Lemma \ref{BoxBound} for large enough $C_1$, with $\kappa = \min(1, d/\alpha)$.  Placing this and \eqref{pBoundPart1} back in \eqref{PartialBound} bounds terms of the form in \eqref{PoissonBoxes} linearly in $n$.

Now, we turn our attention to terms of the form in $\eqref{LBoxes}$.  The procedure is roughly the same, although more complicated.  By applying the Cauchy-Schwarz inequality for sums, and then again for expectations, we have:
\begin{equation*}\label{LBoundCases}
\sum_{\mathcal{B}} \mathbb{E}(L_\mathcal{B}^{2\alpha - 2}p_{\mathcal{B}} \mathbbm{1}_{\{\mathcal{B}~\text{used}\}}(0)) 
\leq \sqrt{\mathbb{E} \sum_{\mathcal{B}} L(\mathcal{B})^{4\alpha - 4} \mathbbm{1}_{\{\mathcal{B}\text{ used}\}}(0)} \sqrt{\mathbb{E} \sum_{\mathcal{B} \in \mathcal{C}} p_{\mathcal{B}}^2}.
\end{equation*}
For the expectation that includes $L(\mathcal{B})^{4\alpha - 4}$, we can use a lattice animal argument from \cite{Howard:1999}.  We can find nearly-exponential tails for the sum of the $L(\mathcal{B})^{4 \alpha - 4}$, as summarized in Lemma \ref{SegmentBound} below, which proves that this expectation is at most linear in $n$.

 Now we turn to bounding the expected sum of the $p_{\mathcal{B}}^2$.  Our procedure here is similar to that for the bound for \eqref{pBoundCases}.  We have:
\begin{equation} \label{PSquaredCases}
\mathbb{E} \sum_{\mathcal{B} \in \mathcal{C}_z} p_{\mathcal{B}}^2 = \mathbb{E} \sum_{\mathcal{B} \in \mathcal{C}_z} p_{\mathcal{B}}^2 \mathbbm{1}_{\{\#\mathcal{C} < C_1n\}} + \mathbb{E} \sum_{\mathcal{B} \in \mathcal{C}_z} p_{\mathcal{B}}^2 \mathbbm{1}_{\{\#\mathcal{C} \geq C_1n\}}
\end{equation}
Like in Equation \eqref{pBoundPart1}, we can bound the first term in \eqref{PSquaredCases} by viewing $\mathcal{C}$ as a lattice animal.  Then, using Theorem \ref{GreedyLatticeThm} and that the Poisson distribution has finite moments of any order, we have:
\begin{equation} \label{OtherGreedyLatticeBound}
\mathbb{E} \sum_{\mathcal{B} \in \mathcal{C}_z} p_{\mathcal{B}}^2 \mathbbm{1}_{\{\#\mathcal{C} < C_1n\}} \leq Cn.
\end{equation}
for some constant, $C$.   And, for the second term in \eqref{PSquaredCases}, we follow the procedure starting at \eqref{BeginPBound}.  Let $R_k$ be the sum of the squared numbers of $Q_n$-points in each unit box contained inside $[-k - 1/2, k + 1/2]^d$. Similarly, let $\tilde R_k$ be the sum of the squared numbers of $Q$-points in each unit box contained inside $[-k - 1/2, k + 1/2]^d$.  For any unit box $\mathcal{B}$,
\[
\Var \tilde R_k = (2k + 1)^d \Var\left(p_{\mathcal{B}}^2\right) = 11(2k + 1)^d, \mbox{ and }
\]
\[
\mathbb{E} \tilde R_k = (2k + 1)^d \mathbb{E} p_{\mathcal{B}}^2 = 2(2k + 1)^d,
\]
which implies that
\[
\mathbb{E}R_k^2 \leq \mathbb{E} \tilde R_k^2 = \Var \tilde R_k + (\mathbb{E} \tilde R_k)^2 \leq 15(2k + 1)^{2d}.
\]
Replacing $P_k$ with $R_k$ in equations \eqref{BeginPBound} through \eqref{LongPoissonBound}, we can conclude that
\[
\mathbb{E} \sum_{\mathcal{B} \in \mathcal{C}} p_{\mathcal{B}}^2 \leq C n
\]
for some constant, $C$, which completes the proof that 
\[
\sum_{\mathcal{B}} \mathbb{E}(L_\mathcal{B}^{2\alpha - 2}p_\mathcal{B} \mathbbm{1}_\mathcal{B}) \leq Cn
\]
for some constant, $C$.
\end{proof}

Now, we turn to the contribution from the Bernoulli random variables.  We have:
\begin{lem} \label{BernoulliBound}
There exists a constant $C > 0$ such that
\[
\mathbb{E}_\pi \sum_{\mathcal{B}} \sum_j \mathbb{E}_\mathcal{B} \left( \Delta_{j, \mathcal{B}} \ol{T}''\right)^2 \leq Cn.
\]
\end{lem}

\begin{proof}
To analyze the inner expectation, we think of all Poisson points outside the box $\mathcal{B}$ and the uniform variables associated to $\mathcal{B}$ as fixed. Then we think of changing the number of points in the box $\mathcal{B}$, adding or removing them at the positions specified by the uniform random variables $U_{\mathcal{B}, 1}, U_{\mathcal{B}, 2}$, and so on.

We will write $\omega_\mathcal{B} = (\omega_{\mathcal{B}, 1}, \omega_{\mathcal{B}, 2}, \ldots)$ for the Bernoulli encoding of the Poisson random variable corresponding to box $\mathcal{B}$, and we write $\varphi(\omega_{\mathcal{B}})$ for the value of the Poisson random variable.  We will write $T''(\varphi(\omega_{\mathcal{B}}))$ for $T''(0, ne_1)$, and we consider what happens when we flip bits in $\omega_{\mathcal{B}}$.  Thus, for $\sigma \in \{0, 1\}^{j - 1}$, the notation $T''(\varphi(\sigma, 1, \omega_{\mathcal{B} > j}))$ will be useful, where the first $j$ terms of $\omega_{\mathcal{B}}$ have been replaced by the binary sequence $\sigma$, and then the digit, $1$.  Let $\mathbf{1}_j$ and $\mathbf{0}_j$ represent the sequence of $j$ ones and $j$ zeroes, respectively.

As before, we write $T''$ for $T''(0, ne_1)$.  Then, let $T_{\mathcal{B}, \mathbf{0}}''$ be the value of $T''$ when all of the points from $\mathcal{B}$ are removed, and let $T_{\mathcal{B}, \infty}''$ be the value of $T''$ when the time to pass between any points in $\mathcal{B}$ is zero (corresponding to having a dense set of Poisson points in the box).  Also, let $\mathbb{E}_{\mathcal{B}^c}$ be the expected value with respect to all variables associated to boxes in the complement $\mathcal{B}^c$, and $\mathbb{E}_{U_{\mathcal{B}}}$ will be the expectation with respect to the uniform random variables inside the box, $\mathcal{B}$.  Also, $\mathbb{E}_{\omega_{\mathcal{B}} > j}$ will be the expectation with respect to the tail part of $\omega_\mathcal{B}$, $(\omega_{\mathcal{B}, j + 1}, \omega_{\mathcal{B}, j + 2}, \ldots)$.  Then, we have:
\begin{align}
\sum_{\mathcal{B}} \sum_j \mathbb{E}_\pi &\left( \Delta_{j, \mathcal{B}} \ol{T}'' \right)^2 \nonumber\\
=~&\sum_{\mathcal{B}} \mathbb{E}_{\mathcal{B}^c} \mathbb{E}_{U_{\mathcal{B}}} \sum_j \mathbb{E}_{\omega_{\mathcal{B}} > j} \frac{1}{2^{j - 1}} \sum_{\sigma \in \{0, 1\}^{j - 1}} (T''(\varphi(\sigma, 1, \omega_{\mathcal{B} > j})) - T''(\varphi(\sigma, 0, \omega_{\mathcal{B} > j})))^2 \nonumber \\
\leq~& \sum_{\mathcal{B}} \mathbb{E}_{\mathcal{B}^c} \mathbb{E}_{U_\mathcal{B}} \sum_j \mathbb{E}_{\omega_{\mathcal{B}} > j} \frac{1}{2^{j - 1}} (T''(\varphi(\mathbf{1}_j, \omega_{\mathcal{B} > j})) - T_{\mathcal{B}, \mathbf{0}}'')^2. \label{PartialPoissonBound}
\end{align}
Here, the last line follows by a telescoping argument, noting that for any sequence of nonincreasing numbers $a_i$, $(a_1 - a_0)^2 + (a_2 - a_1)^2 + \cdots + (a_n - a_{n - 1}) \leq (a_n - a_0)^2$, and that the value of $T''$ is nonincreasing as the value of the binary sequence encoded by $\sigma$ increases.  We also use that $T''(\varphi(\mathbf{1}_i, \omega_{\mathcal{B} > i})) \leq T''(\varphi(\mathbf{0}_i, \omega_{\mathcal{B} > i})) \leq T_{\mathcal{B}, \mathbf{0}}''$.

Define $D_{\mathcal{B}}$ to be one less than the minimum number of Poisson points needed for the box, $\mathcal{B}$, to be used in the $T''$-geodesic between $0$ and $ne_1$.  If the box $\mathcal{B}$ is never used regardless of the number of points in the box, define $D_{\mathcal{B}} = \infty$.  Then, we note that $\big(T''(\varphi(\mathbf{1}_j, \omega_{\mathcal{B} > j})) - T_{\mathcal{B}, \mathbf{0}}''\big)^2 = 0$ off of the event, $\left\{\varphi(\mathbf{1}_j, \omega_{\mathcal{B} > j}) > D_\mathcal{B}\right\}$.  Thus, we upper bound \eqref{PartialPoissonBound} by

\begin{equation}
\sum_{\mathcal{B}}\mathbb{E}_{\mathcal{B}^c} \mathbb{E}_{U_{\mathcal{B}}} (T''_{\mathcal{B},\infty}-T''_{\mathcal{B},\mathbf{0}})^2 \sum_j \frac{1}{2^{j-1}} \mathbb{E}_{\omega_{\mathcal{B}>j}} \mathbbm{1}_{\{\varphi(\mathbf{1}_j,\omega_{\mathcal{B}>j}) > D_{\mathcal{B}}\}}. \label{PartialSmalliBound}
\end{equation}
In the last line, we used the fact that neither $T_{\mathcal{B}, \infty}''$ nor $T_{\mathcal{B}, \mathbf{0}}''$ depends on the variables associated to the box $\mathcal{B}$.  Now, we investigate the sum over $j$.  Let $M_{\omega_\mathcal{B}} = \max\{i : \omega_{\mathcal{B}, 1} = \ldots = \omega_{\mathcal{B}, i} = 1\}$. Then, we have:
\begin{align}
\sum_j \frac{1}{2^{j - 1}} \mathbb{E}_{\omega_\mathcal{B} > j} \mathbbm{1}_{\{\varphi(\mathbf{1}_j, \omega_{\mathcal{B} > j}) > D_\mathcal{B}\}} &= 2\sum_j \mathbb{E}_{\omega_\mathcal{B} \leq j} \mathbbm{1}_{\{\omega_{\mathcal{B}, 1} = 1, \ldots, \omega_{\mathcal{B}, j} = 1\}} \cdot \mathbb{E}_{\omega_\mathcal{B} > j} \mathbbm{1}_{\{\varphi(\mathbf{1}_j, \omega_{\mathcal{B} > j}) > D_\mathcal{B}\}} \nonumber \\
&= 2\sum_j \mathbb{E}_{\omega_\mathcal{B}} \mathbbm{1}_{\{\omega_{\mathcal{B}, 1} = 1, \ldots, \omega_{\mathcal{B}, j} = 1\}} \mathbbm{1}_{\{\varphi(\mathbf{1}_j, \omega_{\mathcal{B} > j}) > D_\mathcal{B}\}} \nonumber \\
&= 2\sum_j \mathbb{E}_{\omega_\mathcal{B}}\mathbbm{1}_{\{\omega_{\mathcal{B}, 1} = 1, \ldots, \omega_{\mathcal{B}, j} = 1\}}  \mathbbm{1}_{\{\varphi(\omega_\mathcal{B}) > D_\mathcal{B}\}} \nonumber \\
&= 2 \mathbb{E}_{\omega_\mathcal{B}} \left[M_{\omega_\mathcal{B}} \mathbbm{1}_{\{\varphi(\omega_\mathcal{B}) > D_\mathcal{B}\}}\right]. \label{SmalliExpBound}
\end{align}
Recall that $\{\varphi(\omega_\mathcal{B}) > D_{\mathcal{B}}\}$ is equivalent to a point in the box $\mathcal{B}$ being used in the $T''$-geodesic, which has the corresponding indicator function $\mathbbm{1}_{\{\mathcal{B} \mbox{\scriptsize\ used}\}}(0)$.  Thus, using \eqref{SmalliExpBound} on the term in \eqref{PartialSmalliBound} yields the upper bound:
\begin{multline}
2 \sum_\mathcal{B} \mathbb{E} \big[(T_{\mathcal{B}, \infty}'' - T_{\mathcal{B}, \mathbf{0}}'')^2 M_{\omega_\mathcal{B}} \mathbbm{1}_{\{\varphi(\omega_\mathcal{B}) > D_{\mathcal{B}}\}}\big]\\
\leq 2 \sqrt{\mathbb{E} \sum_{\mathcal{B}} (T_{\mathcal{B}, \infty}'' - T_{\mathcal{B}, \mathbf{0}}'')^4 \mathbbm{1}_{\{\mathcal{B} \mbox{\scriptsize\ used}\}}(0)} \sqrt{\mathbb{E} \sum_{\mathcal{B}} M_{\omega_\mathcal{B}}^2 \mathbbm{1}_{\{\mathcal{B} \mbox{\scriptsize\ used}\}}(0)}, \label{HereComeLatticeAnimals}
\end{multline}
where the inequality holds by applying the Cauchy-Schwarz inequality twice.  We will bound each square root separately.  First, we will find linear bounds for any integer $p > 1$ for
\begin{multline*}
\mathbb{E} \sum_\mathcal{B} \left|T_{\mathcal{B}, \infty}'' - T_{\mathcal{B}, \mathbf{0}}''\right|^p \mathbbm{1}_{\{\mathcal{B} \mbox{\scriptsize\ used}\}}(0) \\
\leq 2^{p-1} \left[ \mathbb{E} \sum_\mathcal{B} \left|T_{\mathcal{B}, \infty}'' - T''(0, ne_1)\right|^p \mathbbm{1}_{\{\mathcal{B} \mbox{\scriptsize\ used}\}}(0) + \mathbb{E} \sum_\mathcal{B} \left| T''(0, ne_1) - T_{\mathcal{B}, \mathbf{0}}''\right|^p \mathbbm{1}_{\{\mathcal{B} \mbox{\scriptsize\ used}\}}(0)\right].
\end{multline*}
Lemmas \ref{TZeroBound} and \ref{TInfinityBound} tell us exactly that each expectation is at most linear in $n$.  
Thus, it remains to show a linear bound for the last term from \eqref{HereComeLatticeAnimals},
\[
\mathbb{E} \sum_{\mathcal{B}} M_{\omega_\mathcal{B}}^2 \mathbbm{1}_{\{\mathcal{B} \mbox{\scriptsize\ used}\}}(0).
\]
Let $\mathcal{C} = \mathcal{C}(0, ne_1)$ be the set of boxes the $T''$-geodesic from $0$ to $ne_1$ touches.  As in Equation \eqref{pBoundCases},
\begin{equation} \label{MCases}
\mathbb{E} \sum_{\mathcal{B}} M_{\omega_\mathcal{B}}^2 \mathbbm{1}_{\{\mathcal{B} \mbox{\scriptsize\ used}\}}(0) = \mathbb{E} \sum_{\mathcal{B}} M_{\omega_\mathcal{B}}^2 \mathbbm{1}_{\{\mathcal{B} \mbox{\scriptsize\ used}\}}(0) \mathbbm{1}_{\{\# \mathcal{C} \leq C_1 n\}}+ \mathbb{E} \sum_{\mathcal{B}} M_{\omega_\mathcal{B}}^2 \mathbbm{1}_{\{\mathcal{B} \mbox{\scriptsize\ used}\}}(0) \mathbbm{1}_{\{\# \mathcal{C} > C_1 n\}}.
\end{equation}
Since $M_{\omega_\mathcal{B}}$ has a geometric distribution with parameter $p = 1/2$, $M_{\omega_\mathcal{B}}^2$ has finite moments of all orders. As in the discussion after Equation \eqref{pBoundCases} in Lemma \ref{UniformBound}, define a lattice animal to be any face-connected set of unit boxes in $\mathbb{R}^d$.  Let $\mathcal{A}_0(n)$ be the set of lattice animals containing at most $n$ unit boxes, and containing the unit box centered at the origin.  Then, whenever $\# \mathcal{C} \leq C_1n, \mathcal{C} \in \mathcal{A}_0(C_1n)$ almost surely.  Applying Theorem \ref{GreedyLatticeThm} below (originally from \cite{Gandolfi:1994}) yields for some $C$,
\begin{equation} \label{MPartOne}
\mathbb{E} \sum_{\mathcal{B}} M_{\omega_\mathcal{B}}^2 \mathbbm{1}_{\{\mathcal{B} \mbox{\scriptsize\ used}\}}(0) \mathbbm{1}_{\{\# \mathcal{C} \leq C_1 n\}} \leq \mathbb{E} \max_{A \in \mathcal{A}_0(n)} \sum_{\mathcal{B} \in A} M_{\omega_{\mathcal{B}}}^2 \leq Cn.
\end{equation}
Write $\mathcal{C}_k$ for the collection of boxes within distance $k$ of the origin.  As in Equation \eqref{BeginPBound}, we have

\begin{align*}
\mathbb{E} \sum_{\mathcal{B}} M_{\omega_\mathcal{B}}^2 \mathbbm{1}_{\{\mathcal{B} \mbox{\scriptsize\ used}\}}(0) \mathbbm{1}_{\{\# \mathcal{C} \geq C_1 n\}} &\leq \sum_{k = \lceil C_1n \rceil}^\infty \mathbb{E} \mathbbm{1}_{\{\#\mathcal{C} \geq k\}} \sum_{\mathcal{B} \in \mathcal{C}_k} M_{\omega_\mathcal{B}}^2 \nonumber\\
&\leq \sum_{k = \lfloor C_1 \rfloor}^\infty n \mathbb{E} \mathbbm{1}_{\{\#\mathcal{C} \geq kn\}} \sum_{\mathcal{B} \in \mathcal{C}_{(k  + 1)n}} M_{\omega_\mathcal{B}}^2 \nonumber\\
&\leq \sum_{k = \lfloor C_1 \rfloor}^\infty n\sqrt{\mathbb{E} \left( \sum_{\mathcal{B} \in \mathcal{C}_{(k + 1)n}} M_{\omega_\mathcal{B}}^2\right)^2} \sqrt{\mathbb{P}(\#\mathcal{C} \geq kn)}. \label{MPartTwo}
\end{align*}
Using Jensen's inequality for summations, the expectation is bounded by
\[
\#\mathcal{C}_{(k+1)n} \sum_{\mathcal{B} \in \mathcal{C}_{(k+1)n}} \mathbb{E}M_{\omega_\mathcal{B}}^4 \leq C_2[2(k+1)n+1]^{2d},
\]
since $M_{\omega_\mathcal{B}}$ is geometric and has finite fourth moment, $C_2$.  Substituting this into the above sum and using Lemma \ref{BoxBound}, we obtain for large $C_1$,
\[
\mathbb{E} \sum_{\mathcal{B}} M_{\omega_\mathcal{B}}^2 \mathbbm{1}_{\{\mathcal{B} \mbox{\scriptsize\ used}\}}(0) \mathbbm{1}_{\{\# \mathcal{C} \geq C_1 n\}} \leq \sum_{k = \lfloor C_1 \rfloor}^\infty n[2(k+1)n + 1]^d C_3e^{-c_4(kn)^\kappa} \leq C_5e^{-c_6 n^\kappa},
\]
for constants $C_3, c_4, C_5,$ and $c_6$, where $\kappa = \min(1, d/\alpha)$. Thus, this and \eqref{MPartOne} imply that the quantity in \eqref{MCases} is at most linear in $n$.  This completes the proof.
\end{proof}

\section{Proof of Theorem \ref{MainResult}}\label{sec: MainProof}
%
%
%
%
%

Let us begin by stating our variance bound for the averaged passage time, $F$.

\begin{lem} \label{FSublinearVariance}
There exists a constant $C > 0$ such that $\Var F_n \leq Cn/\log n$ for all $n$.
\end{lem}

\begin{proof}
First, we show that $F_n \in L^2$.  We have:
\[
\mathbb{E} F_n^2 = \mathbb{E} \left( \frac{1}{|\Gamma_n|} \sum_{z \in \Gamma_n} T''(z, z + ne_1)\right)^2\leq \frac{1}{|\Gamma_n|} \mathbb{E} \sum_{z \in \Gamma_n} T''(z, z + ne_1)^2\leq \mathbb{E} T''(0, ne_1)^2,
\]
where the last line is due to translation invariance.  But, $T''$ has moments of all orders due to its nearly-exponential tails, as described in Lemma \ref{TTails}.  Thus, $F_n \in L^2$.  This allows us to use Theorem \ref{FSIneq}, the Falik-Samorodnitsky inequality, to conclude:
\begin{equation} \label{MainTheoremPart0}
\Var F_n \log \left[ \frac{\Var F_n}{\sum_i (\mathbb{E} |V_i|)^2} \right] \leq \sum_i \Ent V_i^2.
\end{equation}
From Lemma \ref{SumVBound}, we know that there is a $C > 0$ such that for all $n \geq 1$,
 \begin{equation} \label{MainTheoremPart1}
 \sum_i (\mathbb{E}|V_i|)^2 \leq C n^{1 - 1/4\alpha}.
 \end{equation}
Additionally, from Lemma \ref{EntropyBoundGoal}, we know that there is a $C > 0$ such that for all $n$,
 \begin{equation} \label{MainTheoremPart2}
 \sum_i \Ent V_i^2 \leq Cn.
 \end{equation}
If $\Var F_n \leq n^{1 - 1/8\alpha}$, then the proof is complete.  Alternatively, if $\Var F_n > n^{1 - 1/8\alpha}$, then combining this with Equation \eqref{MainTheoremPart1}, we see there exists $C > 0$ such that for all $n \geq 1$,
\[
\log \left[ \frac{\Var F_n}{\sum_i (\mathbb{E} |V_i|)^2} \right] \geq \log \left[\frac{C n^{1 - 1/8\alpha}}{n^{1 - 1/4\alpha}} \right] \geq C \log n.
\]
Combining this with Equations \eqref{MainTheoremPart0} and \eqref{MainTheoremPart2} gives the desired result,
\[
\Var F_n \leq \frac{Cn}{\log n}.
\]
\end{proof}

Now, we need to convert Lemma \ref{FSublinearVariance} into a statement about the variance of $T$, the original passage time.  First, we compare $T''$, the modified passage time, to $F_n$.

\begin{lem} \label{FtoT} There is some constant $C > 0$ such that for all $n$,
\[
|\Var F_n(0, ne_1) - \Var T''(0, ne_1)| \leq C n^{\frac{3}{4}}.
\]
\end{lem}
\begin{proof}
For any random variable $X$, let $\tilde X := X - \mathbb{E}X$.  Let $\left\lVert \cdot \right\rVert_2$ represent the $L^2$ norm, so $\left\lVert X \right\rVert_2 = \sqrt{\mathbb{E} X^2}$.  We write $\tilde T''$ for $\tilde T''(0,ne_1)$, and likewise for $\tilde F_n$.  Then, we have:
\begin{align}
\left| \Var F_n - \Var T'' \right| &= \left| \left\lVert \tilde F_n' \right\rVert_2 - \left\lVert \tilde T'' \right\rVert_2 \right| \left(\left\lVert \tilde F \right\rVert_2 + \left\lVert \tilde T'' \right\rVert_2\right) \nonumber \\
&\leq  \left\lVert \tilde F_n - \tilde T'' \right\rVert_2 \left(\left\lVert \tilde F \right\rVert_2 + \left\lVert \tilde T'' \right\rVert_2\right) \label{PartialFBound}
\end{align}
We bound the terms in \eqref{PartialFBound} individually.  First, we have from Lemma \ref{TTails} that $T''$ has nearly-exponential tails, so that for some $C > 0$, $\mathbb{E} (T'')^2 \leq C^2 n$.  Thus,
\begin{equation} \label{TBound}
\left\lVert \tilde T'' \right\rVert_2 \leq  \left\lVert T'' \right\rVert_2 \leq C\sqrt{n}, \mbox{ and }
\end{equation}
\begin{align}
\left\lVert \tilde F \right\rVert_2 &= \left\lVert \frac{1}{|\Gamma_n|} \sum_{z \in \Gamma_n} [T''(z, z + ne_1) - \mathbb{E}T''(z, z + ne_1)] \right\rVert_2 \nonumber\\
&\leq \frac{1}{|\Gamma_n|} \sum_{z \in \Gamma_n} \left\lVert T''(z, z + ne_1) - \mathbb{E} T''(z, z + ne_1) \right\rVert_2 \nonumber \\
&\leq  \left\lVert T'' \right\rVert_2 \nonumber \\
&\leq C\sqrt{n}. \label{FBound}
\end{align}
Now, we turn to bounding $\left\lVert \tilde F_n - \tilde T'' \right\rVert_2$.  Because $\mathbb{E} F_n = \mathbb{E} T''$, we have:
\begin{align}
\left\lVert \tilde F_n - \tilde T'' \right\rVert_2 = \left\lVert F_n - T'' \right\rVert_2 &= \left \lVert \frac{1}{|\Gamma_n|} \sum_{z \in \Gamma_n} [T''(z, z + ne_1) - T''(0, ne_1)] \right \rVert_2\nonumber\\
&\leq \frac{1}{|\Gamma_n|} \sum_{z \in \Gamma_n} \left\lVert T''(z, z + ne_1) - T''(0, ne_1)\right\rVert_2. \label{FtoTIneq}
\end{align}
Consider a fixed $z \in \Gamma_n$.  We will analyze the difference, $\left| T''(z, z + ne_1) - T''(0, ne_1)\right|$.  With this in mind, let $r_0, r_1, \ldots, r_{k + 1}$ be the $T''$-geodesic from $0$ to $ne_1$, and let $\tilde r_0, \tilde r_1, \ldots, \tilde r_{j + 1}$ be the $T''$-geodesic from $z$ to $z + ne_1$.  We will need to break into cases depending on whether $k$ and $j$ are zero or nonzero.
First, assume that $j$ is nonzero, so that the $T''$-geodesic from $z$ to $ne_1$ is not a single segment.  Then, we can bound $T''(0, ne_1)$ by following the geodesic from $z$ to $z + ne_1$:
\[
T''(0, ne_1) \leq \phi_n \left( \left \lVert 0 - \tilde r_1 \right \rVert \right) + T''\left(\tilde r_1, \tilde r_j\right) + \phi_n\left( \left \lVert \tilde r_j - ne_1 \right \rVert \right).
\]
We can apply the modified triangle inequality for $\phi_n$, Lemma \ref{PhiIneq}, and recognize that $T''\left( \tilde r_1, \tilde r_j \right) \leq T''(z, z + ne_1)$ to obtain:
\begin{align*}
T''(0, ne_1) \leq & 2^\alpha \big[ \phi_n \left( \left\lVert 0 - z \right \rVert \right) + \phi_n \left( \left \lVert z - \tilde r_1 \right \rVert \right) \big] + T''(z, z + ne_1)\\
&+ 2^\alpha \big[ \phi_n \left( \left \lVert \tilde r_j - (z + ne_1) \right \rVert \right) + \phi_n \left( \left \lVert z+ ne_1 - ne_1 \right \rVert \right) \big].
\end{align*}
If $j = 0$, then
\[
T''(0, ne_1) \leq \phi_n(\lVert ne_1 - 0 \rVert) = T''(z, z +  ne_1),
\]
which only improves the inequality from above.  Through a symmetrical analysis, we can get similar bounds for $T''(z, z + ne_1)$ compared to the $T''$-geodesic from $0$ to $ne_1$.  Combining the inequalities gives the bound:
\begin{align}
\left| T''(0, ne_1) - T''(z, z + ne_1) \right| \leq &\ 2^{\alpha + 2} \phi_n \left( \left \lVert z \right \rVert \right) + 2^\alpha \big[ \phi_n \left( \left \lVert z - \tilde r_1 \right \rVert \right) + \phi_n \left( \left \lVert \tilde r_j - (z + ne_1) \right \rVert \right)\big] \nonumber\\
& + 2^\alpha \big[\phi_n \left( \left \lVert 0 - r_1 \right \rVert \right) + \phi_n \left( \left \lVert r_k - ne_1 \right \rVert \right) \big]. \label{TDiffInequality}
\end{align}
(Here it is understood that the $\tilde r_i$ and $r_i$ terms are present only in the cases where $j$ or $k$ are nonzero.) Except for the $\phi_n\left(\left\lVert z \right\rVert\right)$ term, each $\phi_n$ term corresponds to the length of some segment in the $T''$-geodesic between $0$ and $ne_1$ or $z$ and $z +ne_1$.  So, let $L_{\mbox{\scriptsize max}} = \max \lVert r_i - r_{i + 1} \rVert$.  Then, from the definition of $\phi_n$, we have that $\phi_n \left( \left \lVert 0 - r_1 \right \rVert \right) \leq \phi_n(L_{\mbox{\scriptsize max}}) \leq L_{\mbox{\scriptsize max}}^\alpha$.  Therefore, using Minkowski's inequality and Equation \eqref{TDiffInequality}, we have:
\[
\left\lVert T''(0, ne_1) - T''(z, z + ne_1) \right\rVert_2 \leq 2^{\alpha + 2}  \phi_n(\|z\|) + 4 \cdot2^\alpha \left \lVert L_{\mbox{\scriptsize max}}^\alpha\right \rVert_2.
\]
Now, because $z \in \Gamma_n := \left\{z \in \mathbb{Z}^d : \lVert z \rVert_\infty \leq n^{\frac{1}{4\alpha}}\right\}$, we can leverage our previous choice of power of $n$ in this definition to obtain:
\begin{equation} \label{DiameterChoice}
\phi_n \left(\lVert z \rVert\right) \leq \phi_n \left(\sqrt{d} n^{1/4\alpha}\right) \leq d^{\alpha/2} n^{1/4}.
\end{equation}
Additionally, by Lemma \ref{MaxLengthBound} below, there exists a $C > 0$ such that:
\[
\left \lVert L_{\mbox{\scriptsize max}}^\alpha\right \rVert_2 \leq Cn^{1/4}.
\]
Thus, for some larger $C > 0$ and for any $z \in \Gamma_n$,
\[
\left\lVert T''(0, ne_1) - T''(z, z + ne_1) \right\rVert_2 \leq Cn^{1/4}.
\]
Plugging this into Equation \eqref{FtoTIneq} gives:
\[
\left\lVert \tilde F_n - \tilde T'' \right\rVert_2 \leq \frac{1}{\left| \Gamma_n \right|} \sum_{z \in \Gamma_n} Cn^{1/4} = Cn^{1/4}.
\]
Substituting this along with \eqref{TBound} and \eqref{FBound} into \eqref{PartialFBound} completes the proof.
\end{proof}

Now that we have shown that we can bound the variance of $T''$ in terms of the variance of $F$, we would also like to show that the variances of $T''$ and $T$ are close, so that we were justified in approximating $T''$ with $T$.  The following lemma will complete our proof of Theorem \ref{MainResult}:

\begin{lem} \label{TdoubletoT}
There exists a constant, $C > 0$, such that for all $n$,
\[
\Var T(0, ne_1) \leq 2 \Var T''(0, ne_1) + C.
\] 
\end{lem}

\begin{proof}
In order to prove this, we rely on Lemmas \ref{TDifferences} and \ref{TTails} from \cite{Howard:2001}, reproduced below.  We will write $T$ for $T(0, ne_1)$, and will do similarly for $T'$ and $T''$.  (The passage time $T'$ is defined above Lemma~\ref{TDifferences} and is the same as $T$ except that points are added at $a$ and $b$.) Then, for any integer $p > 0$,
\begin{align*}
\mathbb{E}\left[\left|(T'')^p - (T')^p\right|\right] = \mathbb{E} \left[ \left|(T'')^p - (T')^p\right| \mathbbm{1}_{T' \neq T''}\right] &\leq \sqrt{\mathbb{E} \left[ (T'')^p - (T')^p\right]^2 \mathbb{E} \mathbbm{1}_{T' \neq T''}}\\
&\leq \sqrt{2 \left[ \mathbb{E}(T'')^{2p} + \mathbb{E}(T')^{2p}\right] \mathbb{P}(T' \neq T'')}.
\end{align*}
In the first inequality, we used the Cauchy-Schwarz inequality.  Now, because of the nearly-exponential tails of $T'$ and $T''$ from Lemma \ref{TTails}, we have that $\mathbb{E}(T')^{2p} \leq C n^{2p}$ and $\mathbb{E}(T'')^{2p} \leq C n^{2p}$ for some constant, $C$.  Additionally, $\mathbb{P}(T' \neq T'')$ decays nearly exponentially in $n$ according to Lemma \ref{TDifferences}.  Combining these two yields that for some $C_1, c_2,$ and $c_3 > 0,$
\[
\mathbb{E}\left[\left|(T'')^p - (T')^p\right|\right] \leq C_1 \exp\left(-c_2 n^{c_3}\right).
\]
Using this expression for $p = 1$ and $p = 2$ along with the exponential tails for $T'$ and $T''$ yields:
\begin{align} 
\left|\Var T'' - \Var T'\right| &= \left| \left[\mathbb{E}(T'')^2 - \mathbb{E}(T')^2\right] - \left[ \left(\mathbb{E} T''\right)^2 - \left(\mathbb{E} T'\right)^2\right]\right|\nonumber \\
&\leq \mathbb{E}\left[\left|(T'')^2 - (T')^2\right|\right] + \mathbb{E}\left|T'' - T'\right| \cdot \left(\mathbb{E}T'' + \mathbb{E}T'\right)\nonumber \\
& \leq C_1 \exp\left(-c_2 n^{c_3}\right),\label{T''toT'}
\end{align}
for possibly bigger $C_1$, and smaller $c_2$ and $c_3$.  Now, we must compare the variance of $T'$ to the variance of $T$.  Note that for general random variables $X$ and $Y$, 
\[
\Var (X+Y) \leq 2\Var X + 2 \Var Y.
\]
Thus, choosing $X = T'$ and $Y = T - T'$ yields for some $C > 0$,
\begin{equation} \label{T'toT}
\Var T \leq 2 \Var T' + 2\Var(T - T') \leq 2 \Var T' + C.
\end{equation}
This last inequality follows from the comparison between $T$ and $T'$ in Lemma \ref{TDifferences}.  Combining Equations \eqref{T''toT'} and \eqref{T'toT} and perhaps changing the constant yields the desired result.
\end{proof}

\section{Bounds on boxes, numbers of Poisson points, and lengths} \label{GreedyLattice}

Throughout the paper, we have needed results to bound the total number of boxes a geodesic passes through, or the total number of points in boxes the geodesic uses.  The proofs of these results will use lattice animals.  Let us begin with the total number of boxes that a $T''$-geodesic touches.

Recall that $\mathbb{R}^d$ was broken into unit boxes, ordered $\mathcal{B}_1, \mathcal{B}_2, \ldots$. Let $\mathcal{C}(x, y)$ be the set of boxes touched by the $T''$-geodesic from $x$ to $y$.  Then, let $\# \mathcal{C}(x, y)$ be the number of boxes in $\mathcal{C}(x, y)$.  We will need to know information on $\mathcal{C}(0, ne_1)$, and in addition, on the number of boxes any possible $T''$-geodesic within the averaged regions defined by $\Gamma_n$ could touch.  The following lemma applies to both cases: in the lemma when $\gamma = 1$, we get bounds on $\#\mathcal{C}(0, ne_1)$, and when $\gamma = 1/4$, we get bounds on the number of boxes within an averaged region.

  \begin{lem} \label{BoxBound}
 Let $\gamma \in (0,1]$ and $C_3 > 0$ be constants, and let $\kappa = \min(1, d/\alpha)$.  Let $\mathcal{B}(x)$ be the unit box containing the point $x$. Then, there exist positive constants $C_1$ and $C_2$ depending on $\gamma$ and $C_3$ such that for all $n \geq 1$, for any $x$ and $y$ with $\lVert x - y \rVert \leq C_3 n^\gamma$, and for all $\ell \geq C_1 {n^\gamma},$
  \[
  \mathbb{P} \left( \max_{v \in \mathcal{B}(x), w \in \mathcal{B}(y)} \# \mathcal{C}(v, w) \geq \ell \right) \leq C_1 \exp\left(-C_2 \ell^\kappa\right).
  \]
  \end{lem}

\begin{proof}
Let $\mathring{x} \in \mathcal{B}(x)$ and $\mathring{y} \in \mathcal{B}(y)$ be any pair of points where $\#\mathcal{C}(v, w)$ attains its maximum, so that
\[
\# \mathcal{C}(\mathring{x}, \mathring{y}) =  \max_{v \in \mathcal{B}(x), w \in \mathcal{B}(y)} \# \mathcal{C}(v, w).
\]
To bound $\# \mathcal{C}(\mathring{x}, \mathring{y})$, we will break into cases: either the $T''$-geodesic from $\mathring{x}$ to $\mathring{y}$ passes through a high proportion of boxes with many points, or the $T''$-geodesic has a large passage time.  We will show that the probability of each of these cases is small.  With this in mind, let $\mathcal{G}$ be the set of ``good'' unit boxes $\mathcal{B}$ where the total number of Poisson points in $\mathcal{B}$ and any boxes touching $\mathcal{B}$ (including at corners) is at most $C_4$ for some large constant $C_4$. Then, for any constant $c_5$ (to be made small later),
\begin{multline} \label{BoxCases}
\mathbb{P}\left(\# \mathcal{C}(\mathring{x}, \mathring{y}) \geq \ell \right) =  \mathbb{P}\bigg(\# \mathcal{C}(\mathring{x}, \mathring{y}) \geq \ell \mbox{ and } \#\big[\mathcal{C}(\mathring{x}, \mathring{y}) \cap \mathcal{G}\big] < c_5 \# \mathcal{C}(\mathring{x}, \mathring{y}) \bigg)
\\ + \mathbb{P}\bigg(\# \mathcal{C}(\mathring{x}, \mathring{y}) \geq \ell \mbox{ and } \#\big[\mathcal{C}(\mathring{x}, \mathring{y}) \cap \mathcal{G}\big] \geq c_5 \#\mathcal{C}(\mathring{x}, \mathring{y}) \bigg).
\end{multline}
We bound each term separately.  For the first term on the right side of \eqref{BoxCases}, we will use a lattice animal argument. Call two boxes connected if they share a face. We will show that it is unlikely for any connected set of boxes $\mathcal{C}$ of cardinality at least $\ell \geq C_1 {n^\gamma}$ to have fewer than $c_5\# \mathcal{C}(\mathring{x}, \mathring{y})$ good boxes.  Let $\mathfrak{M}_m$ be the collection of all connected sets of boxes containing $\mathring{x}$ with $m$ boxes total.  Then,
\begin{align} 
\mathbb{P}\bigg(\# \mathcal{C}(\mathring{x}, \mathring{y}) \geq \ell \mbox{ and } \#\big[\mathcal{C}(\mathring{x}, \mathring{y}) \cap \mathcal{G}\big] < c_5 \# \mathcal{C}(\mathring{x}, \mathring{y}) \bigg)
&\leq \sum_{m = \lceil \ell \rceil}^\infty \mathbb{P}\left(\min_{\mathcal{C} \in \mathfrak{M}_m} \#(\mathcal{C} \cap \mathcal{G}) < c_5 m\right) \nonumber \\
&\leq \sum_{m = \lceil \ell \rceil}^\infty \sum_{\mathcal{C} \in \mathfrak{M}_m} \mathbb{P}\left(\#(\mathcal{C} \cap \mathcal{G}) < c_5 m\right) \nonumber \\
&= \sum_{m = \lceil \ell \rceil}^\infty \sum_{\mathcal{C} \in \mathfrak{M}_m} \mathbb{P}\left( \sum_{\mathcal{B} \in \mathcal{C}} X(\mathcal{B}) < c_5 m\right), \label{BoxToBernoulli}
\end{align}
where $X(\mathcal{B})$ is the indicator of the event that $\mathcal{B} \in \mathcal{G}$.  Note that the probability $p = p(C_4)$ of $X(\mathcal{B})$ being $1$ depends on the constant $C_4$.  To bound the probabilities in this summation, we use that the $X(\mathcal{B})$'s are 2-dependent.  To use this independence, consider partitioning the unit boxes in $\mathbb{R}^d$ into $3^d$ groups of boxes, $P_1, P_2, \ldots, P_{3^d}$, so that for any two boxes in the same $P_i$, the annuli of boxes surrounding them are disjoint.  Then, let $P_{\mathcal{C}}$ be the set from the partition that contains the most boxes from $\mathcal{C}$, so that by the pigeonhole principle, $\#(P_{\mathcal{C}} \cap \mathcal{C})$ is at least $\left \lceil \frac{\#\mathcal{C}}{3^d}\right \rceil = \left \lceil \frac{m}{3^d} \right \rceil$.  Choosing any $k := \left \lceil \frac{m}{3^d}\right \rceil$ boxes in $\#(P_{\mathcal{C}} \cap \mathcal{C})$ and writing them as $\mathcal{B}_{j_1}, \mathcal{B}_{j_2}, \ldots, \mathcal{B}_{j_k}$, we have for any $\lambda > 0$:
\begin{align}
\mathbb{P}\left( \sum_{\mathcal{B} \in \mathcal{C}} X(\mathcal{B}) < c_5 m\right) \leq \mathbb{P}\left(\sum_{i = 1}^{k} X(\mathcal{B}_{j_i}) < c_5m\right)
&= \mathbb{P} \left(e^{-\lambda \sum_{i = 1}^k X(\mathcal{B}_{j_i})} > e^{-\lambda c_5m}\right) \nonumber\\
&\leq e^{\lambda c_5 m} \mathbb{E} e^{-\lambda \sum_{i = 1}^k X(\mathcal{B}_{j_i})} \nonumber\\
&= e^{\lambda c_5m} \left(\mathbb{E}e^{-\lambda X(\mathcal{B}_{j_1})}\right)^k. \label{ExpGF}
\end{align}
Since $X(\mathcal{B}_{j_1})$ is a Bernoulli random variable with parameter $p$, $\mathbb{E}e^{-\lambda X(\mathcal{B}_{j_1})} = 1 - p + pe^{-\lambda}$, which is always less than $1$.  Picking up at \eqref{ExpGF}, and recalling that $k = \left \lceil \frac{m}{3^d}\right \rceil$, we have
\begin{align*}
e^{\lambda c_5m} \left(\mathbb{E}e^{-\lambda X(\mathcal{B}_{j_1})}\right)^k &\leq e^{\lambda c_5m} \left(1 - p + pe^{-\lambda}\right)^{\frac{m}{3^d}}\nonumber \\
&= e^{m\left(\lambda c_5 + \frac{1}{3^d} \ln (1 - p + pe^{-\lambda})\right)}
\end{align*}
It is well-known that the number of lattice animals of size $m$ grows exponentially in $m$: from Equation (4.24) of \cite{Grimmett:1999}, we have that $|\mathfrak{M}_m| \leq 7^{dm}$ for all $m$. Thus, we have the following:
\begin{equation*}
\sum_{\mathcal{C} \in \mathfrak{M}_m} \mathbb{P}\left(\sum_{\mathcal{B} \in \mathcal{C}} X(\mathcal{B}) < c_5 m\right) \leq e^{m\left(\lambda c_5 + \frac{1}{3^d}\ln(1-p+pe^{-\lambda}) + d\ln 7\right)}.\label{BoxesExpBound}
\end{equation*}
By choosing $c_5$ small, $C_4$ large (so that $p$ is close to 1) and $\lambda$ large, the exponent can be made negative. Specifically, fix $c_5 \leq \frac{1}{4\cdot 3^d}$, $\lambda \geq (2d\ln 7)4\cdot 3^d$, and then $C_4$ to be so large that $\ln(1-p+pe^{-\lambda}) < -\lambda/2$.  Therefore, for some positive constants $C_6$ and $c_7$,
\[
\mathbb{P}\left(\min_{\mathcal{C} \in \mathfrak{M}_m} \#(\mathcal{C} \cap \mathcal{G}) < c_5 m\right) \leq C_6e^{-c_7m}.
\]
We use this fact in the top line of \eqref{BoxToBernoulli} to conclude for some positive constant $C_8$,
\begin{align}
\mathbb{P}\bigg(\# \mathcal{C}(\mathring{x}, \mathring{y}) \geq \ell \mbox{ and } \#\big[\mathcal{C}(\mathring{x}, \mathring{y}) \cap \mathcal{G}\big] < c_5 \# \mathcal{C}(\mathring{x}, \mathring{y}) \bigg) &\leq \sum_{m = \lceil \ell \rceil}^\infty C_6e^{-c_7m}\nonumber\\
&\leq C_8 e^{-c_7 \ell}. \label{BBound1}
\end{align}

This finishes our bound for the first term on the right side of \eqref{BoxCases}.  For the second term, we will show that when a $T''$-geodesic passes through many good boxes, it must be long, and this is unlikely.  Consider any good box $\mathcal{B} \in \mathcal{G}$ that is not touching or equal to the boxes containing $\mathring{x}$ or $\mathring{y}$, and is touched by the $T''$-geodesic.  In order to touch $\mathcal{B}$, the $T''$-geodesic from $\mathring{x}$ to $\mathring{y}$ must pass through an annulus around $\mathcal{B}$ that contains at most $C_4$ points.  This means that the $T''$-geodesic travelled Euclidean distance at least $1$ in at most $C_4$ steps. For such a box $\mathcal{B}$, define $b_{\mathcal{B}}$ to be the first Poisson point the $T''$-geodesic uses from $\mathcal{B}$ (if it exists). Then, define $a_{\mathcal{B}}$ to be the first Poisson point the $T''$-geodesic uses prior to $a_{\mathcal{B}}$ that is not in the annulus of unit boxes surrounding $\mathcal{B}$ (if it exists). Then, the smallest contribution to $T''(a_\mathcal{B}, b_{\mathcal{B}})$ comes when the points are equally-spaced, and we have that
\[
T''(a_{\mathcal{B}}, b_{\mathcal{B}}) \geq C_4^{1 - \alpha}.
\]
Note that even when there are no existing points $a_{\mathcal{B}}$ and $b_{\mathcal{B}}$ that are used by the $T''$-geodesic, $C_4^{1-\alpha}$ still gives a lower bound for the length of the portion of the $T''$-geodesic passing through the annulus around $\mathcal{B}$.  Now, we find a lower bound for the length of the $T''$-geodesic based on all of the good boxes it passes through.  To guarantee that the segments between $a_{\mathcal{B}}$ and $b_{\mathcal{B}}$ are disjoint (except possibly at endpoints), we again partition the unit boxes in $\mathbb{R}^d$ into $3^d$ different groups $P_1, P_2, \ldots, P_{3^d}$, so that any two boxes in the same $P_i$ are surrounded by annuli of boxes that are disjoint.  If $\# \mathcal{C}(\mathring{x}, \mathring{y}) \geq \ell$ and $\#(\mathcal{C}(\mathring{x}, \mathring{y}) \cap \mathcal{G}) \geq c_5 \# \mathcal{C}(\mathring{x}, \mathring{y}) \geq c_5\ell$, then by the pigeonhole principle at least one part $P_{\mathcal{C}}$ in the partition has at least $c_5\ell/3^d$ boxes in it.  Thus, except for perhaps boxes adjacent to $\mathring{x}$ or $\mathring{y}$, each of the good boxes in this partition corresponds to a segment of the $T''$-geodesic of length at least $C_4^{1 - \alpha}$, and we can conclude:
\begin{equation} \label{T-lowerbound}
T''(\mathring{x}, \mathring{y}) \geq \left(c_5\ell/3^d - 2 \cdot 3^d \right)C_4^{1 - \alpha} \geq c_9 \ell
\end{equation}
for some $c_9 > 0$ and $C_1$ sufficiently large, because $\ell \geq C_1 n^\gamma \geq C_1$.

Let $\tilde x$ and $\tilde y$ be the (deterministic) centers of the boxes containing $\mathring{x}$ and $\mathring{y}$.  Then, let $r_0 = \tilde x, r_1, \ldots, r_k, r_{k + 1} = \tilde y$ be the $T''$-geodesic between $\tilde x$ and $\tilde y$.  Assuming $k > 0$, we can use the modified triangle inequality for $\phi_n$ in Lemma \ref{PhiIneq} to obtain:
\begin{align*}
T''\left(\mathring{x}, \mathring{y}\right) &\leq \phi_n \left( \left \lVert \mathring{x} - r_1 \right \rVert \right) + T''(r_1, r_k) + \phi_n \left( \left\lVert r_k - \mathring{y} \right \rVert \right)\\
&\leq 2^\alpha \left[ \phi_n \left( \left \lVert \mathring{x} - \tilde x \right \rVert \right) + \phi_n \left( \left \lVert \tilde x - r_1 \right \rVert \right) \right] + T''(r_1, r_k)\\
&\ \ + 2^\alpha \left[ \phi_n \left( \left \lVert \mathring{y} - \tilde y \right \rVert \right) + \phi_n \left( \left \lVert r_k - \tilde y \right \rVert \right)\right]\\
&\leq 2^\alpha T''(\tilde x, \tilde y) + 2^\alpha \left[ \phi_n \left( \left\lVert \mathring{x} - \tilde x \right \rVert \right) + \phi_n \left( \left \lVert \mathring{y} - \tilde y \right\rVert \right)\right].
\end{align*}
The last two $\phi_n$ terms are both bounded by a constant, because $\mathring{x}$ and $\tilde x$ are within the same unit box, and $\mathring{y}$ and $\tilde y$ are also within the same unit box.  On the other hand, if $k = 0$, then
\begin{align*}
T''(\mathring{x}, \mathring{y}) &\leq 2^{2\alpha} \left[ \phi_n\left( \left \lVert \mathring{x} - \tilde x \right \rVert \right) + \phi_n \left( \left \lVert \tilde x - \tilde y \right \rVert \right) + \phi_n \left( \left \lVert \tilde y - \mathring{y} \right \rVert \right) \right]\\
&= 2^{2\alpha}T''(\tilde x, \tilde y) + 2^{2\alpha}\left[ \phi_n \left( \left\lVert \mathring{x} - \tilde x \right \rVert \right) + \phi_n \left( \left \lVert \mathring{y} - \tilde y \right\rVert \right)\right],
\end{align*}
which is similar to the previous bound.  Thus, Equation \eqref{T-lowerbound} implies:
\begin{align*}
T''(\tilde x, \tilde y) &\geq 2^{-2\alpha} c_9 \ell -  C_{10}\\
&\geq  c_{11} \ell
\end{align*}
for sufficiently large $C_1$ and for some $C_{10}, c_{11} > 0$, because $\ell \geq C_1 n^\gamma$ and $n \geq 1$.  So, we have shown that if $\# \mathcal{C}(\mathring{x}, \mathring{y}) \geq \ell$ and $\#(\mathcal{C}(\mathring{x}, \mathring{y}) \cap \mathcal{G}) \geq c_5 \# \mathcal{C}(\mathring{x}, \mathring{y})$, then $T''(\tilde x, \tilde y) \geq c_{11} \ell$ for some constant, $c_{11}$.  Thus, by making $C_1$ sufficiently large, we can use Lemma \ref{TprimeTails} to conclude that for some $C_{12}$ and $c_{13} > 0$ and with $\kappa = \min(1, d/\alpha)$,
\begin{align}
 \mathbb{P}\bigg(\# \mathcal{C}(\mathring{x}, \mathring{y}) \geq \ell \mbox{ and } \#\big[\mathcal{C}(\mathring{x}, \mathring{y}) \cap \mathcal{G}\big] \geq c_5 \#\mathcal{C}(\mathring{x}, \mathring{y}) \bigg) &\leq \mathbb{P}(T''(\tilde x, \tilde y) \geq c_{11} \ell) \nonumber\\
 &\leq C_{12} \exp\left(-c_{13} \ell^\kappa\right). \label{BBound2}
\end{align}
Plugging \eqref{BBound1} and \eqref{BBound2} into \eqref{BoxCases} completes the proof.
\end{proof}

Now, we find a bound on the maximum length of a segment on a $T''$-geodesic.
\begin{lem} \label{MaxLengthBound}
Define $L_{\mbox{\scriptsize max}} = \max_i \lVert r_i - r_{i + 1} \rVert$ for consecutive points $r_i$ and $r_{i + 1}$ in the $T''$-geodesic between $0$ and $ne_1$.  Then, for any $p > 0$ and any $\gamma > 0$, there exists a constant $C_{\gamma, p}$ such that for all $n \geq 1$,
\[
\mathbb{E} L_{\mbox{\scriptsize max}}^p \leq C_{\gamma, p} n^\gamma.
\]
\end{lem}
\begin{proof}
Let $\mathcal{C} = \mathcal{C}(0, ne_1)$ be the set of boxes touched by the $T''$-geodesic from $0$ to $ne_1$, and let $\#\mathcal{C} = \# \mathcal{C}(0, ne_1)$ be the number of boxes in this set.  We have:
\begin{equation} \label{MaxLengthCases}
\mathbb{E} L_{\mbox{\scriptsize max}}^p = \mathbb{E} L_{\mbox{\scriptsize max}}^p \mathbbm{1}_{\{\# \mathcal{C} < Cn\}} + \sum_{k = Cn}^\infty \mathbb{E} L_{\mbox{\scriptsize max}}^p \mathbbm{1}_{\{\# \mathcal{C} = k\}}.
\end{equation}

Let us bound the first term.  Suppose that $\#\mathcal{C}(0, ne_1) \leq Cn$, but $L_{\mbox{\scriptsize \color{black}max}}^p > tn^\gamma$ for some $t \geq C_0$, with $C_0 > 0$ a large constant restricted above Equation \eqref{A-Bound} below.  We will show this implies there is a large region with no Poisson points, which is unlikely.  As in \cite{Howard:2001}, define:
\[
\mathcal{W}_{\phi_n}(x, y) = \{a: \phi_n(\lVert x - a\rVert) + \phi_n(\lVert a - y\rVert) \leq \phi_n(\lVert x - y\rVert)\}.
\]
In words, $\mathcal{W}_{\phi_n}(x, y)$ is the set of points that shorten the direct path between $x$ and $y$.  Lemma 5.1 of \cite{Howard:2001} tells us that $\mathcal{W}_{\phi_n}(a, b)$ is closed and convex.  Combining this convexity with Lemma \ref{W-dimensions} below yields that there is a small constant $c_1 > 0$ such that for all $x$ and $y$ with $\lVert x - y \rVert \geq 1$, $\mathcal{W}_{\phi_n}(x, y)$ contains a $d$-dimensional box (with sides parallel to the axes) with side length $c_1 \sqrt{\lVert x - y \rVert}$.

So, if $L_{\mbox{\scriptsize max}}^p > tn^\gamma$, let $r$ and $s$ be endpoints of a segment with $\lVert r - s \rVert \geq t^{1/p}n^{\gamma/p}$.  Then, $\mathcal{W}_{\phi_n}(r, s)$ contains a box of side length $c_1t^{1/2p}n^{\gamma/2p}$. Consider dividing $\mathbb{R}^d$ into boxes of side length $\frac{c_1}{2} t^{1/2p} n^{\gamma/2p}$, centered at points of $\left(\frac{c_1}{2} t^{1/2p} n^{\gamma/2p} \mathbb{Z}\right)^d$.  If $L_{\mbox{\scriptsize \color{black}max}}^p > tn^\gamma$, one of these deterministic boxes must contain no Poisson points from $Q_n$.  Moreover, if $\# \mathcal{C} \leq Cn$, then the $T''$-geodesic from $0$ to $ne_1$ must be contained in the region $[-Cn, Cn]^d$.  Therefore there is a $\frac{c_1}{2} t^{1/2p} n^{\gamma/2p}$-box that overlaps with $[-Cn, Cn]^d$ and does not intersect $Q_n$.

With this in mind, let $E_n$ be the event that there is a $\frac{c_1}{2} t^{1/2p} n^{\gamma/2p}$-box overlapping $[-Cn, Cn]^d$ that contains no Poisson points in $Q_n$.  Note that there are at most $(2Cn)^d/\left(\frac{c_1}{2} t^{1/2p} n^{\gamma/2p}\right)^d$ boxes of side length $\frac{c_1}{2} t^{1/2p} n^{\gamma/2p}$ contained completely within $[-Cn, Cn]^d$, meaning that in total we can bound the number of such boxes intersecting $[-Cn, Cn]^d$ by $C_2 n^{d(1 - \gamma/2p)}$ for some $C_2 > 0$ (since $t \geq C_0$). Now, let $A_1$ be the event that the $\frac{c_1}{2} t^{1/2p} n^{\gamma/2p}$-box centered at the origin contains no points from $Q_n$.  A unit box has no points from $Q_n$ if and only if it has no points from $Q$.  So, let $A_2$ be the event that the $\frac{c_1}{4} t^{1/2p} n^{\gamma/2p}$-box centered at the origin has no points from $Q$.  Assuming that $C_0$ is large enough in the restriction $t \geq C_0$, the $\frac{c_1}{4} t^{1/2p} n^{\gamma/2p}$-box is contained within the union of the unit boxes contained entirely within the $\frac{c_1}{2} t^{1/2p} n^{\gamma/2p}$-box, so that $\mathbb{P}(A_1) \leq \mathbb{P}(A_2)$. Because our Poisson process has rate $1$, we have that $\mathbb{P}(A_2) = \exp\left(-c_3 t^{d/2p} n^{\gamma d/2p}\right)$ for some $c_3 > 0$, so
\begin{equation}\label{A-Bound}
\mathbb{P}(A_1) \leq \exp\left(-c_3 t^{d/2p} n^{\gamma d/2p}\right).
\end{equation}
Then, because the Poisson points in each box are independent, we have:
\begin{align}
\mathbb{P}(L_{\mbox{\scriptsize max}}^p \geq tn^\gamma \mbox{ and } \# \mathcal{C} \leq n) &\leq \mathbb{P}(E_n) \leq 1 - (1 - \mathbb{P}(A_1))^{C_2 n^{d(1 - \gamma/2p))}}. \label{LmaxCaseOne}
\end{align}
Next, we use the inequalities $\ln(1-x)\geq -2x$ for small $x \geq 0$ and $e^{-x} \geq 1-x$ for all $x$ to get
 for all $t \geq C_0$ and for sufficiently large $n$,
\begin{align*}
(1 - \mathbb{P}(A_1)^{C_2 n^{d(1 - \gamma/2p))}} &\geq \exp\big[ \ln(1 - \exp(-c_3t^{d/2p} n^{d \gamma/2p})) \cdot C_2 n^{d(1 - \gamma/2p)} \big]\\
&\geq \exp\big[ - 2 \exp(-c_3t^{d/2p} n^{d \gamma/2p}) C_2 n^{d(1 - \gamma/2p)}\big]\\
&\geq 1 - 2\exp(-c_3t^{d/2p} n^{d \gamma/2p}) C_2 n^{d(1 - \gamma/2p)}.
\end{align*}
Substituting this into equation \eqref{LmaxCaseOne} yields for all $t \geq C_0$ and $n$,
\begin{align*}
\mathbb{P}(L_{\mbox{\scriptsize max}}^p \geq tn^\gamma \mbox{ and } \# \mathcal{C} \leq Cn) &\leq 2 \exp(-c_3t^{d/2p} n^{d \gamma/2p}) C_2 n^{d(1 - \gamma/2p)}\\
&\leq C_4\exp(-c_3t^{d/2p} n^{d\gamma/4p}),
\end{align*}
where $C_4 > 0$ is some large constant.  We can use this bound to conclude:
\begin{align}
\mathbb{E} L_{\mbox{\scriptsize max}}^p \mathbbm{1}_{\{\# \mathcal{C} < Cn\}} &= n^\gamma \int_0^\infty \mathbb{P}\left( \frac{L_{\text{max}}^p}{n^\gamma} \geq t, \#\mathcal{C} < Cn\right)~\text{d}t \nonumber \\
&\leq  C_0 n^\gamma + n^\gamma \int_{C_0}^\infty \mathbb{P}\left( L_{\text{max}}^p \geq t n^\gamma, \#\mathcal{C} < Cn\right)~\text{d}t \nonumber \\
&\leq \frac{C_\gamma}{2} n^\gamma, \label{Part1LMax}
\end{align}
for all $n$, and for $C_\gamma$ large (to be restricted further momentarily).  Now, we turn to the second part of equation \eqref{MaxLengthCases}.  We note that if $\# \mathcal{C} = k$, then the maximum segment is of length at most $k\sqrt{d}$.  Using the Cauchy-Schwarz inequality and then Lemma \ref{BoxBound} (assuming $C$ is sufficiently large) with $\kappa = \min(1, d/\alpha)$, we have:
\begin{align}
\sum_{k = Cn}^\infty \mathbb{E} L_{\mbox{\scriptsize max}}^p \mathbbm{1}_{\{\# \mathcal{C} = k\}} &\leq \sum_{k = Cn}^\infty \sqrt{\mathbb{E}  L_{\mbox{\scriptsize max}}^{2p} \mathbbm{1}_{\{\# \mathcal{C} = k\}}} \sqrt{\mathbb{P} (\# \mathcal{C} = k)} \nonumber \\
&\leq \sum_{k = Cn}^\infty C_5 \left(k\sqrt{d}\right)^p  \exp\left(-c_6k^\kappa\right) \nonumber \\
&\leq \frac{C_\gamma}{2} n^\gamma \label{Part2LMax}
\end{align}
for all $n$, for constants $C_5$ and $c_6$, and for $C_\gamma$ sufficiently large.  Plugging equations \eqref{Part1LMax} and \eqref{Part2LMax} into equation \eqref{MaxLengthCases} finishes the proof.

\end{proof}

There are several times in the proofs above where we need to bound the difference between $T''$ in the original environment and $T''$ in an environment in which points have been added or removed from $Q_n$.  Here, $T''_{\mathcal{B}, \mathbf{0}}(0, ne_1)$ represents the passage time of the $T''$-geodesic from $0$ to $n e_1$ when all points from $\mathcal{B}$ have been removed, and $T''_{\mathcal{B}, \infty}(0, ne_1)$ when it takes no time to pass between points in $\mathcal{B}$ (corresponding to a dense set of Poisson points in $\mathcal{B}$).  Recall that $\mathbbm{1}_{\{\mathcal{B} \mbox{\scriptsize\ used}\}}(0)$ is the indicator of the event that a point from $\mathcal{B}$ is used in the $T''$-geodesic from $0$ to $ne_1$.

\begin{lem} \label{TZeroBound}
For any integer $p \geq 1$, there is a $C > 0$ such that for all $n \geq 1$,
\[
\mathbb{E} \sum_{\mathcal{B}} \left[ T''_{\mathcal{B}, \mathbf{0}}(0, ne_1) - T''(0, n e_1) \right]^p \mathbbm{1}_{\{\mathcal{B} \mbox{\scriptsize\ used}\}}(0) < C n.
\]
\end{lem}

\begin{lem} \label{TInfinityBound}
For any integer $p \geq 1$, there is a $C > 0$ such that for all $n \geq 1$,
\[
\mathbb{E} \sum_{\mathcal{B}} \left[ T''(0, n e_1) - T''_{\mathcal{B}, \infty}(0, n e_1) \right]^p \mathbbm{1}_{\{\mathcal{B} \mbox{\scriptsize\ used}\}}(0) < C n.
\]
\end{lem}

Before proving these lemmas, we will need statements about the sums of powers of lengths of segments in our $T''$-geodesic, or lengths of segments near (but not on) the $T''$-geodesic.  The following two statements summarize these results, and are mild extensions of arguments from \cite{Howard:2001}.  The intuition behind the proof structure is as follows: if a segment in a $T''$-geodesic is large, this implies that there is a large region where there are no Poisson points.  Although this may happen for some segments of a $T''$-geodesic, it is unlikely to happen for many sections of the $T''$-geodesic simultaneously, as shown by some lattice animal arguments.  The statements are as follows:

\begin{lem} \label{SegmentBound}
Let the $T''$-geodesic from $0$ to $n e_1$ be denoted by the points $(r_1, r_2, \ldots, r_N)$, and let $L_k$ denote the Euclidean length of the $k$th segment, $L_k := \lVert r_k - r_{k + 1} \rVert$.  Then, for any $p > 1$, there exist positive constants $C_1, c_2$, and $c_3$ such that:
\[
\mathbb{P}\left[ \sum_{k = 1}^{N - 1} L_k^p > x \right] \leq C_1 \exp\left(-c_2x^{c_3}\right) \mbox{ for all } x \geq C_1 n.
\]
\end{lem}

This lemma is identical to Equation (3.10) in \cite{Howard:2001}, except that the power $p$ was $2\alpha$ in the paper.  This makes no difference in their proof, which is not reproduced here.  However, we will prove the similar statement, Lemma \ref{OffSegmentBound}, below.

In this paper, we also need similar results bounding the sums of lengths of segments \emph{near} the $T''$-geodesic.  The reason for this is that we consider resampling the Poisson points in unit boxes, which may add points to the box.  There is no simple way to bound the amount that a $T''$-geodesic can be shortened when these points are added, in terms of the original $T''$-geodesic.  Instead, we will have to argue that any segment near the $T''$-geodesic cannot be large, and since these segments could be used in the resampled environment, this will imply that the change in the $T''$-geodesic from adding points is small.

More precisely, for any unit box, $\mathcal{B}$, define $\mathcal{Q}_{\mathcal{B}}$ to be the points $q \in Q_n \backslash \mathcal{B}$ closest to $\mathcal{B}$ in the following sense: there is a $T''$-geodesic from $q$ to some point $x \in \partial \mathcal{B}$ that equals the segment from $q$ to $x$.  Note that here, $x$ need not be a Poisson point in $Q_n$, because the $T''$ distance adds a point at $x$.  Then, define
\[
L_{\mathcal{B}}^\infty = \max_{q \in \mathcal{Q}_{\mathcal{B}}} \left[ \min_{x \in \partial \mathcal{B}} \lVert q - x \rVert \right].
\]

\begin{lem} \label{OffSegmentBound} For any $p > 1$, there exist positive constants $C_1, C_2,$ and $C_3$ such that:
\[
\mathbb{P}\left[ \sum_{\mathcal{B}} \left[(L_\mathcal{B}^\infty)^p \mathbbm{1}_{\{\mathcal{B}\mbox{\scriptsize\ used}\}}(0)\right] > x \right] \leq C_1 \exp\left(-C_2x^{C_3}\right) \mbox{ for all } x \geq C_1 n.
\]
\end{lem}

Now that we stated the main results we will need, we begin with the proofs of these lemmas, starting with Lemma \ref{TZeroBound}.

\begin{proof}[Proof of Lemma \ref{TZeroBound}]
Similarly to the proof of Lemma \ref{MaxV}, let $s_{\mathcal{B}}^-$ be the first point the $T''$-geodesic uses from $\mathcal{B}$, and let $r_\mathcal{B}^-$ be the point immediately preceding $s_\mathcal{B}^-$.  Likewise, let $s_\mathcal{B}^+$ be the last point the $T''$-geodesic uses from $\mathcal{B}$, and let $r_\mathcal{B}^+$ be the next point the $T''$-geodesic uses after $s_\mathcal{B}^+$.  If $\mathcal{B}$ contains $0$, let $r_\mathcal{B}^- = s_\mathcal{B}^- = 0$, and if $\mathcal{B}$ contains $ne_1$, let $r_\mathcal{B}^+ = s_\mathcal{B}^+ =  ne_1$. Then, we can see that $ T''_{\mathcal{B}, \mathbf{0}}(0, ne_1) \leq T''(0, ne_1) + \phi_n(\lVert r_\mathcal{B}^- - r_\mathcal{B}^+ \rVert)$.  By applying the modified triangle inequality for $\phi_n$, Lemma \ref{PhiIneq}, we obtain the following:
\begin{multline}
\sum_\mathcal{B} \mathbb{E} \big[ T''_{\mathcal{B}, \mathbf{0}}(0, ne_1) - T''(0, ne_1) \big]^p  \leq \sum_\mathcal{B} \mathbb{E} \big[ \phi_n(\lVert r_\mathcal{B}^- - r_\mathcal{B}^+\rVert) \mathbbm{1}_{\{\mathcal{B}\mbox{\scriptsize\ used}\}}(0) \big]^p \\
\leq 2^{2\alpha p} \mathbb{E} \sum_\mathcal{B} \big[ \phi_n(\lVert r_\mathcal{B}^- - s_\mathcal{B}^-\rVert) + \phi_n(\lVert s_\mathcal{B}^- - s_\mathcal{B}^+\rVert) + \phi_n(\lVert s_\mathcal{B}^+ - r_\mathcal{B}^+\rVert)  \big]^p \mathbbm{1}_{\{\mathcal{B}\mbox{\scriptsize\ used}\}}(0). \label{UsedPhiIneq}
\end{multline}

Now, all of the $\phi_n(\lVert s_\mathcal{B}^- - s_\mathcal{B}^+\rVert)$ terms are bounded by a constant $C$, since each box $\mathcal{B}$ has side length one.  Because of the indicator function, these terms correspond to segments of the optimal path, where each segment appears at most twice (in case a situation occurs like $r_{\mathcal{B}_1}^- = s_{\mathcal{B}_2}^+$ for two different boxes $\mathcal{B}_1$ and $\mathcal{B}_2$).  As before, let $\# \mathcal{C} = \# \mathcal{C}(0, ne_1)$ be the number of boxes the $T''$-geodesic from $0$ to $ne_1$ touches.  Additionally, define $L_k$ as the Euclidean length of the $k$th segment of the $T''$-geodesic.  Then, by noticing that $(a + b)^p \leq 2^p(a^p + b^p)$ for $a, b > 0$ and also that $\phi_n(r) \leq r^\alpha$ for all $r$, we can bound the expression in \eqref{UsedPhiIneq} by the following:
\begin{equation*} \label{TwoLinearParts} 
2^{2\alpha p + 2p + 1} \left[ \mathbb{E} \sum_k L_k^{\alpha p} + C\mathbb{E} \# \mathcal{C} \right].
\end{equation*}
We have a linear bound for the sum of $L_k^{\alpha p}$ by Lemma \ref{SegmentBound}.  The nearly-exponential tails from Lemma \ref{BoxBound} show that $\mathbb{E} \# \mathcal{C}$ must be linear as well, completing the proof.
\end{proof}

\begin{proof}[Proof of Lemma \ref{TInfinityBound}]
This proof is very similar to the proof of Lemma \ref{TZeroBound}.  Consider the $T''$-geodesic in the new environment in which it takes no time to pass between points in the box $\mathcal{B}$.  If the original $T''$-geodesic passed through $\mathcal{B}$, then the $T''$-geodesic in the new environment does as well.  In this case, let $r_{\infty, \mathcal{B}}^-$ be the last point in this new $T''$-geodesic before it touches $\mathcal{B}$, and let $s_{\infty, \mathcal{B}}^-$ be the point in $\mathcal{B}$ that it touches.  Similarly, let $s_{\infty, \mathcal{B}}^+$ be the point in $\mathcal{B}$ where the $T''$-geodesic leaves, and let $r_{\infty, \mathcal{B}}^+$ be the next point in the path.  When $\mathcal{B}$ contains $0$, we let $r_{\infty, \mathcal{B}}^- = s_{\infty, \mathcal{B}}^- = 0.$  When $\mathcal{B}$ contains $ne_1$, we let $r_{\infty, \mathcal{B}}^+ = s_{\infty, \mathcal{B}}^+ = ne_1$.  
Then, $[T''(0, ne_1) - T_{\mathcal{B}, \infty}''(0, ne_1)]\mathbbm{1}_{\{\mathcal{B}\mbox{\scriptsize\ used}\}}(0) \leq \phi_{n}(\lVert r_{\infty, \mathcal{B}}^- - r_{\infty, \mathcal{B}}^+ \rVert)\mathbbm{1}_{\{\mathcal{B}\mbox{\scriptsize\ used}\}}(0)$.  Recall that $\#\mathcal{C} = \#\mathcal{C}(0, ne_1)$ is the number of boxes the $T''$-geodesic from $0$ to $ne_1$ touches.  Again, using Lemma \ref{PhiIneq} and following the same reasoning as in Lemma \ref{TZeroBound}, we get:
\begin{align*}
&\sum_{\mathcal{B}} \mathbb{E}[T''(0, ne_1) - T_{\mathcal{B}, \infty}''(0, ne_1)]^p\mathbbm{1}_{\{\mathcal{B}\mbox{\scriptsize\ used}\}}(0) \\
\leq~& \sum_\mathcal{B} \mathbb{E}\left[\phi_n(\lVert r_{\infty, \mathcal{B}}^- - r_{\infty, \mathcal{B}}^+ \rVert) \mathbbm{1}_{\{\mathcal{B}\mbox{\scriptsize\ used}\}}(0)\right]^p\\
\leq~& 2^{2\alpha p} \mathbb{E} \sum_\mathcal{B} \big[ \phi_n(\lVert r_{\infty, \mathcal{B}}^- - s_{\infty, \mathcal{B}}^-\rVert) + \phi_n(\lVert s_{\infty, \mathcal{B}}^- - s_{\infty, \mathcal{B}}^+\rVert) + \phi_n(\lVert s_{\infty, \mathcal{B}}^+ - r_{\infty, \mathcal{B}}^+\rvert)  \big]^p \mathbbm{1}_{\{\mathcal{B}\mbox{\scriptsize\ used}\}}(0)\\
\leq~& 2^{2\alpha p + 2p + 1} \left[ \mathbb{E} \sum_\mathcal{B} (L_\mathcal{B}^\infty)^{\alpha p} \mathbbm{1}_{\{\mathcal{B} \mbox{\scriptsize\ used}\}}(0) + C\mathbb{E}\#{\mathcal{C}}\right].
\end{align*}
The first expectation is at most linear due to the exponential tails from Lemma \ref{OffSegmentBound}, and the second one is at most linear from Lemma \ref{BoxBound}, which completes the proof.
\end{proof}

\begin{proof}[Proof of Lemma \ref{OffSegmentBound}]
We follow the proof of Equation (3.10) in \cite{Howard:2001}.  First, we will split the sum of $(L_\mathcal{B}^\infty)^p$ into three pieces which are easier to analyze.  To do so, we will consider whether or not the path from $0$ to $ne_1$ uses points from many boxes. Additionally, we break down the lattice into four different, nested grids.  We have already used a grid of unit boxes $\mathcal{B}$, with vertices at points $\mathbb{Z}^d + (1/2, \ldots, 1/2)^d$.  Within that, we also have an $\epsilon/3^n$ grid from the definition of $Q_n$, where $\epsilon = 1/k$ for a positive odd integer $k$, so that the $\epsilon/3^n$ boxes nest evenly within the unit boxes.  This grid is chosen so that it is unlikely that an $\epsilon/3^n$ box has more than one Poisson point inside of it.  In \cite{Howard:2001}, the authors use this small grid size to deduce that $T$ and $T''$ behave similarly to each other, as described in Lemma \ref{TTails} below.  We will also use the intermediate $\epsilon$-grid, and we will consider the number of $\epsilon$-boxes that a $T''$-geodesic touches.  Equation (3.22) from \cite{Howard:2001} (reproduced as Equation \eqref{MEpsilonBound} below), states that it is unlikely that a $T''$-geodesic touches many $\epsilon$-boxes.  Finally, we will also need to consider a $\lambda$-grid, for $\lambda > 1$ a large odd integer (so that the unit boxes nest inside the $\lambda$ grid).  $\lambda$ is chosen to be large enough that the probability of a $\lambda$-box being empty is below the critical probability for site percolation on the nearest neighbor $\mathbb{Z}^d$ lattice.  This is a requirement of using Theorem 6.4.2 of \cite{Howard:1999}, which we use in Equation \eqref{xiPartTwo} below.

With this in mind, we will define a refined version of $\# \mathcal{C}(0, ne_1)$ from above, which was the number of unit boxes touched by the $T''$-geodesic from $0$ to $ne_1$.  Instead, consider counting the number of $\epsilon$-boxes along the $T''$-geodesic as follows: let $\beta_1$ be the $\epsilon$-box that covers $0$.  Then, for $i \geq 1$ let $\beta_{i+1}$ be the $\epsilon$-box that the $T''$-geodesic enters after it leaves $\beta_{i}$ for the last time (if such a box exists).  Continue counting boxes in this manner, and let $M(ne_1)$ be the total number of such $\epsilon$-boxes along the $T''$-geodesic from $0$ to $ne_1$.  Then, for any $\ell > 0$ and $h_0 > 0$ sufficiently large from the definition of $\phi_n$, define:
\begin{equation} \label{SBreakdown}
\sum_{\mathcal{B}} \left[(L_\mathcal{B}^\infty)^p \mathbbm{1}_{\{\mathcal{B}\mbox{\scriptsize\ used}\}}(0)\right] = S_1 + S_2 + S_3
\end{equation}
with
\begin{align*}
S_1 &= \sum_{\mathcal{B} : L_\mathcal{B}^\infty \leq h_0} \left[(L_\mathcal{B}^\infty)^p \mathbbm{1}_{\{\mathcal{B}\mbox{\scriptsize\ used}\}}(0)\right],\\
S_2 &= \mathbbm{1}_{\{M(ne_1) \geq \ell\}} \sum_{\mathcal{B} : L_\mathcal{B}^\infty > h_0} \left[(L_\mathcal{B}^\infty)^p \mathbbm{1}_{\{\mathcal{B}\mbox{\scriptsize\ used}\}}(0)\right],\\
\mbox{and}\\
S_3 &= \mathbbm{1}_{\{M(ne_1) < \ell\}} \sum_{\mathcal{B} : L_\mathcal{B}^\infty > h_0} \left[(L_\mathcal{B}^\infty)^p \mathbbm{1}_{\{\mathcal{B}\mbox{\scriptsize\ used}\}}(0)\right].
\end{align*}

First, $S_1 \leq h_0^p \cdot \# \mathcal{C}(0, ne_1)$, which has nearly-exponential tails according to Lemma \ref{BoxBound}.  Then, by Equation (3.22) in \cite{Howard:2001} (which relies on $\epsilon$ being sufficiently small), $M(ne_1)$ has nearly-exponential tails. So, for $\kappa = \min(1, d/\alpha)$, there are positive constants $C_1$ and $c_2$ such that:
\begin{equation}\label{MEpsilonBound}
\mathbb{P}(S_2 > \ell) \leq \mathbb{P}(M(ne_1) \geq  \ell) \nonumber \leq C_1 \exp(-c_2 \ell^{\kappa}) \mbox{ for all } \ell \geq C_1 n. 
\end{equation}

Now, handling the tails of $S_3$ is where our argument must diverge slightly from the argument in \cite{Howard:2001}.  When considering sums of $L_k^p$ in Lemma \ref{SegmentBound}, the key was to note that $L_k^p$ could only be large when a large region devoid of Poisson points intersected with the $T''$-geodesic.  However, for $L_\mathcal{B}^\infty$ to be large, we cannot guarantee that there is a large empty region \emph{on} the $T''$-geodesic, and we instead will argue that there is a large empty region \emph{near} the $T''$-geodesic.  Once we have shown this, we can use a lattice animal argument to find exponential tails.

To begin, we will consider a single box, $\mathcal{B}$, and show that $L_\mathcal{B}^\infty$ can be large only if at least one of a set of deterministic boxes is empty.  In order to do this, we consider what region would be empty when the $T''$-geodesic between points $a$ and $b$ is the direct segment $\overline{ab}$.  As per before, define:
\[
\mathcal{W}_{\phi_n}(a, b) = \{c \in \mathbb{R}^d : \phi_n(\lVert a - c\rVert) + \phi_n(\lVert c - b \rVert) \leq \phi_n(\lVert a - b \rVert)\}.
\]
Then, Lemma 5.1 of \cite{Howard:2001} tells us that $\mathcal{W}_{\phi_n}(a, b)$ is closed and convex.  Additionally, Lemma 5.4 from \cite{Howard:2001}, reproduced as Lemma \ref{HN:Regions} below, guarantees that the regions grow at least linearly in volume with respect to $\lVert b - a \rVert$.  However, we will need something slightly stronger: that for large $\lVert b - a \rVert$, the regions $\mathcal{W}_{\phi_n}(a, b)$ (for large $n$) contain a box of arbitrarily large (but constant) width within a fixed distance of $a$. This is implied by the convexity of $\mathcal{W}_{\phi_n}(a, b)$, the statistical isotropy of the Poisson process, and Lemma \ref{E-Regions} below. As illustrated in Figure \ref{fig:E-Rectangle}, due to convexity,
\begin{figure}
\begin{center}
\includegraphics[width=.6\textwidth]{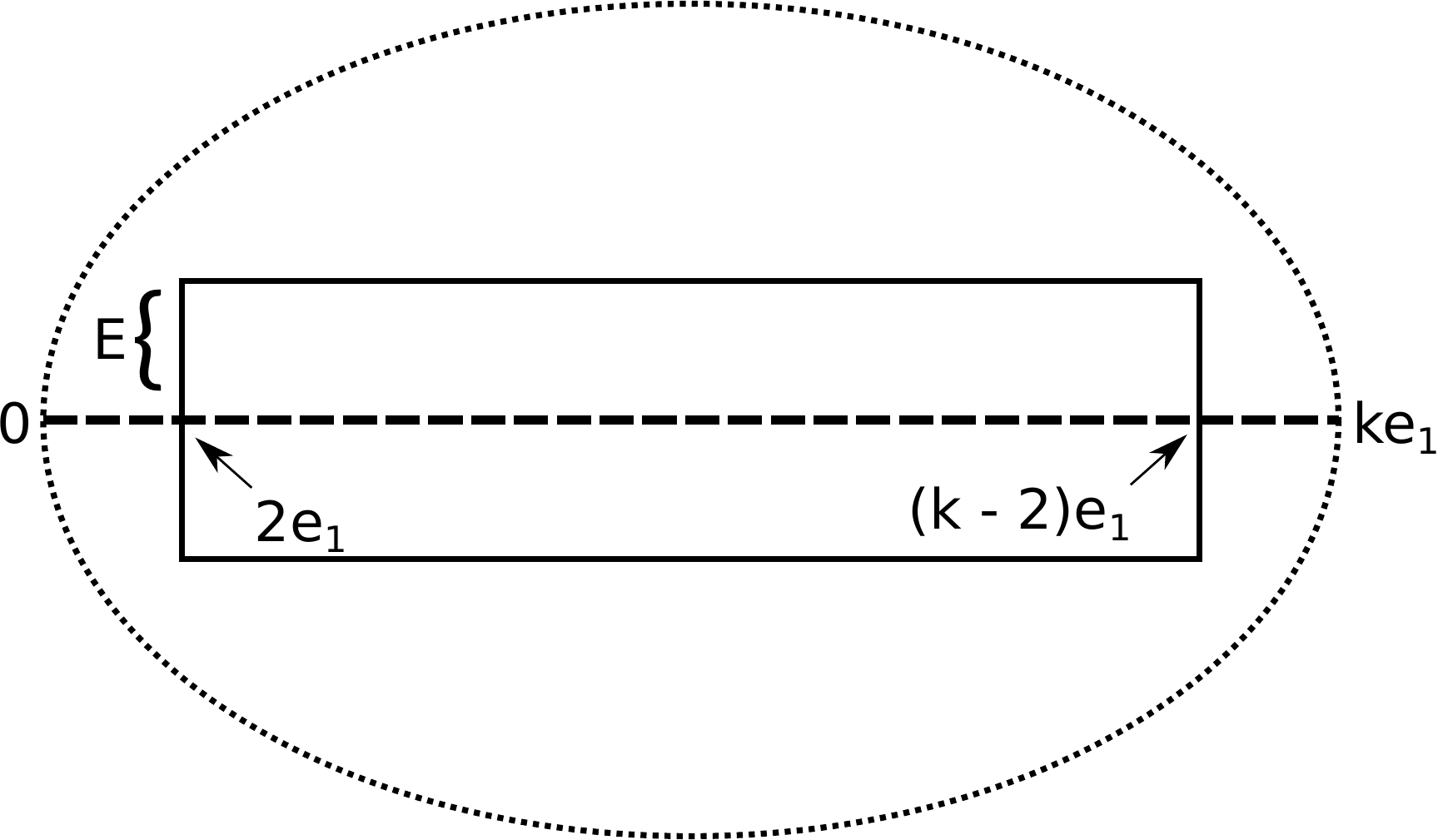}
\end{center}
\caption{Here, the region $\mathcal{W}_{\phi_n}(0, ke_1)$ is illustrated for some large value of $k$.  The region is convex, and thus by Lemma \ref{E-Regions} below, for any $E > 0$, if $k$ is sufficiently large, it must contain the illustrated rectangle with height $2E$ and length $k - 2$.  By symmetry, we can find a similar box in any $\mathcal{W}_{\phi_n}(a, b)$ when $\lVert b - a \rVert$ is sufficiently large.}
\label{fig:E-Rectangle}
\hrule
\end{figure}
 the lemma implies that for any constant $E > 0$ and sufficiently large $n$ and $k$, there is a rectangle with height $2E$ and length $k - 4$ contained inside $\mathcal{W}_{\phi_n}(0, ke_1)$.  Due to symmetry, this can be translated to any points $a$ and $b$, as long as $\lVert b - a \rVert$ is sufficiently large.
 
Now, consider the $\lambda$-grid in $\mathbb{R}^d$.  Choose $E > \sqrt{d} \lambda$, to be restricted further later. Then, the rectangle with height $2E$ contained inside $\mathcal{W}_{\phi_n}(0, ke_1)$ must contain a collection of $\lambda$-boxes that are connected by $(d-1)$-dimensional faces, since $2E > 2\sqrt{d}\lambda$, twice the length of the diagonal of a $\lambda$-box.  Let $\Lambda(a, b)$ be the face-connected set of $\lambda$-boxes contained inside $\mathcal{W}_{\phi_n}(a, b)$ that includes this face-connected set of boxes in the rectangle, and let $\#\Lambda(a, b)$ be the number of such $\lambda$-boxes.  Then, $\#\Lambda(a, b) \geq D_E \lVert b - a \rVert$ for some constant $D_E > 0$ that can be made arbitrarily large as $E$ increases.

Let $C_E$ be a constant such that whenever $\lVert a - b \rVert > C_E$, the rectangle from Figure \ref{fig:E-Rectangle} (translated so that it is along the line segment $\ol{ab}$) is inside $\mathcal{W}_{\phi_n}(a, b)$.  Note that if $\lVert a - b \rVert > C_E$ for $C_E$ sufficiently large, then the box containing the point $3\sqrt{d} \lambda$ units along the line segment $\ol{ab}$ must be in $\Lambda(a, b)$.  A unit line segment can touch at most $d + 1$ unit boxes: starting with one box, the path can touch more boxes corresponding to its movement in each of the $d$ dimensions.  Therefore, the portion of the segment starting at $a$ with length $3\sqrt{d}\lambda$ can touch at most $3\sqrt{d}(d + 1)$ total $\lambda$-boxes. Thus if we define a box path between any two boxes to be a collection of face-adjacent boxes that connect the two boxes, and define the box distance between two boxes to be the one less than the minimum number of boxes in a box path between them, then there is a $\lambda$-box in $\Lambda(a,b)$ within box-distance $3\sqrt{d}(d + 1)$ of the $\lambda$-box covering $a$.
 
If $L_\mathcal{B}^\infty > h_0$, then there is a pair of points $a \in \partial \mathcal{B}$ and $b \in Q_n$ such that $\lVert b - a \rVert > h_0$ and $\mathcal{W}_{\phi_n}(a, b)$ contains no points from $Q_n$.  But, by our reasoning above, this implies that if $h_0$ is large enough, there is a $\lambda$-box within box-distance $3\sqrt{d}(d + 1)$ of the $\lambda$-box containing $a$ that has no Poisson points, and this box is contained in a cluster of empty $\lambda$-boxes $\Lambda(a, b)$ inside $\mathcal{W}_{\phi_n}(a, b)$ such that $\#\Lambda(a, b) \geq D_E \lVert b - a \rVert$.
 
 With this in mind, let $\xi$ denote the collection of $\lambda$-boxes that contain an $\epsilon$-box counted by $M(ne_1)$, along with any $\lambda$-box that is within box-distance $3\sqrt{d}(d + 1)$ of these boxes.  For any $\lambda$ box, $\nu$, let $\# \Lambda_\nu$ represent the number of $\lambda$-boxes in the nearest-neighbor cluster of empty boxes centered at $\nu$.  For each large $L_\mathcal{B}^\infty$, the value of $\#\Lambda(a, b)$ corresponding to this segment equals $\#\Lambda_\nu$ for some $\nu \in \xi$. However, the same $\lambda$-box $\nu$ may correspond to multiple different unit boxes $\mathcal{B}$ used in the geodesic.  As a crude upper bound, note that there are $\lambda^d$ unit boxes within any $\lambda$-box, and the number of $\lambda$ boxes within box-distance $3\sqrt{d}(d + 1)$ of a given $\lambda$-box is at most $(6\sqrt{d}(d + 1) + 1)^d < (7\sqrt{d}(d + 1))^d$.  With this in mind, let $C_\lambda = (7\sqrt{d}(d + 1)\lambda)^d$.  Then, we have: 
 
 \begin{align*}
 S_3 &= \mathbbm{1}_{\{M(ne_1) < \ell\}} \sum_{\mathcal{B} : L_\mathcal{B}^\infty > h_0} \left[(L_\mathcal{B}^\infty)^p \mathbbm{1}_{\{\mathcal{B}\mbox{\scriptsize\ used}\}}(0)\right]\\
 &\leq \mathbbm{1}_{\{M(ne_1) < \ell\}} \sum_{\nu \in \xi} C_\lambda D_E^{-p} (\# \Lambda_\nu)^p.
 \end{align*}
 
 Now, we know that $|\xi|$, the number of boxes in $\xi$, cannot be more than $C_d M(ne_1)$ for some constant $C_d$ depending on the dimension.  Let $\Xi$ be the set of all lattice animals in $\mathbb{Z}^d$ containing the origin, so that $\xi$ can be viewed as an element in $\Xi$.  Thus, following the proof in \cite{Howard:2001}, we can make the division for any $0 < \gamma < 1$:
 \begin{align*}
 \{S_3 > \ell \} &\subset \left\{ |\xi| < \ell^\gamma \mbox{ and } \sum_{\nu \in \xi} (\# \Lambda_\nu)^p > \frac{D_E^p}{C_\lambda} \ell\right\}\\
 & \ \ \ \ \cup \left\{\ell^\gamma \leq |\xi| \leq C_d \ell \mbox{ and } \sum_{\nu \in \xi} (\# \Lambda_\nu)^p > \frac{D_E^p}{C_\lambda} \ell\right\}.
 \end{align*}
 When $|\xi| < \ell^\gamma$ and $\sum_{\nu \in \xi} (\# \Lambda_\nu)^p > \frac{D_E^p}{C_\lambda} \ell$, then by the pigeonhole principle, there is some $\nu \in \xi$ such that $(\#\Lambda_\nu)^p > \frac{D_E^p\ell}{C_\lambda|\xi|} > \frac{D_E^p}{C_\lambda} \ell^{1 - \gamma}$.  On the other hand, when $\ell^\gamma \leq |\xi| \leq C_d\ell$ and $\sum_{\nu \in \xi} (\# \Lambda_\nu)^p > \frac{D_E^p}{C_\lambda} \ell$, then $\frac{1}{|\xi|} \sum_{\nu \in \xi} (\#\Lambda_\xi)^p > \frac{D_E^p\ell}{C_\lambda |\xi|} > \frac{D_E^p}{C_\lambda C_d}$.  Thus,
 \begin{align*}
 \{S_3 > \ell \} & \subset \left\{ \exists \nu \subset [-\lambda \ell^\gamma, \lambda\ell^\gamma]^d \cap \mathbb{Z}^d \mbox{ with }  \# \Lambda_\nu > \frac{D_E}{C_\lambda^{1/p}} \ell^{(1 - \gamma)/p}\right\}\\
 & \ \ \ \ \cup \left\{\exists \xi \in \Xi \mbox{ with } |\xi| \geq \ell^\gamma \mbox{ and } \frac{1}{|\xi|} \sum_{\nu \in \xi} (\# \Lambda_\nu)^p > \frac{D_E^p}{C_\lambda C_d} \right\}.
 \end{align*}
 
Recall that $\lambda$ is large enough that the probability of a $\lambda$-box being empty is below the critical probability for site percolation on $\mathbb{Z}^d$. Then from Theorem 6.75 of \cite{Grimmett:1999}, we can conclude that there exists a $c_4 > 0$ such that for all $\ell$, $\mathbb{P}(\# \Lambda_\nu > \ell) \leq e^{-c_4\ell}$.  In turn, this implies:
\begin{equation*} \label{xiPartOne}
\mathbb{P}\left[ \exists \nu \subset [-\lambda \ell^\gamma, \lambda \ell^\gamma]^d \cap \mathbb{Z}^d \mbox{ with }  \# \Lambda_\nu > D_E \ell^{(1 - \gamma)/p} \right]
\leq (2\ell^\gamma + 1)^d \exp\left(-c_4 D_E \ell^{(1 - \gamma)/p}\right).
\end{equation*}
Next, by Theorem~6.4.2 of \cite{Howard:1999} on lattice animals, if $D_E$ is sufficiently large (in terms of the dimension, $d$, which is possible by taking $E$ sufficiently large), then we have for some $C_5 > 0$ and a possibly smaller $c_4 > 0$ that:
\begin{equation} \label{xiPartTwo}
\mathbb{P}\left[ \exists \xi \in \Xi \mbox{ with } |\xi| \geq \ell^\gamma \mbox{ and } \frac{1}{|\xi|} \sum_{\nu \in \xi} (\# \Lambda_\nu)^p > \frac{D_E^p}{C_d}\right]
\leq C_5 \exp(-c_4\ell^{\gamma/p}).
 \end{equation}
Thus, we conclude that for all $x$ and appropriate constants $C_1, c_2,$ and $c_3$,
 \[
 \mathbb{P}(S_3 > \ell) \leq C_1 \exp\left(-c_2\ell^{c_3}\right).
 \]
 We have found nearly-exponential tails for $S_1, S_2,$ and $S_3$ in Equation \eqref{SBreakdown} and this completes the proof.
\end{proof}
We need to know information about the regions
\[
\mathcal{W}_{\phi_n}(a, b) := \left\{c \in \mathbb{R}^d : \phi_n(\lVert c - a \rVert) + \phi_n(\lVert c - b \rVert) \leq \phi_n(\lVert a - b \rVert)\right\}.
\]
Lemma 5.1 of \cite{Howard:2001} tells us that these regions are closed and convex, and Lemma 5.4, reproduced as Lemma \ref{HN:Regions} below, tells us that they grow at least linearly in volume with respect to $\lVert b - a \rVert$.  However, we will need slightly more specific versions of their lemmas.  In particular, we must know that the regions always contain a box of constant size a constant distance from their endpoints as $\lVert b - a \rVert$ grows, summarized in Lemma \ref{E-Regions}.  This comes up in the analysis of $T''_{\mathcal{B}, \infty}(0, ne_1)$, where we need to analyze regions devoid of Poisson points near (but not in) the box $\mathcal{B}$.  Additionally, we will need that the regions grow polynomially in all dimensions, summarized in Lemma \ref{W-dimensions}.  This will help us analyze $L_{\mbox{\scriptsize max}}$ in Lemma \ref{MaxLengthBound}, because we will know that a $T''$-geodesic has a large jump only when there is a large region devoid of Poisson points.

\begin{lem} \label{E-Regions}
Let $e_1$ and $e_2$ be the unit vectors in the first and second coordinate directions of $\mathbb{R}^d$.  For any $E > 0$, there is a $D > 0$ such that $2e_1 + Ee_2 \in \mathcal{W}_{\phi_n}(0, ke_1)$ for all $h_0, k > D$, where $h_0$ is the constant in the definition of $\phi_n$.
\end{lem}
\begin{proof}
We must show that for all $k$ and $h_0$ sufficiently large,
\[
\phi_n\left(\sqrt{4 + E^2}\right) + \phi_n\left(\sqrt{(k - 2)^2 + E^2}\right) \leq \phi_n(k).
\]
If $h_0$ is large, we can assume that $\sqrt{4 + E^2} \leq h_0 \leq h_n$, the cutoff in the piecewise function $\phi_n$, because $h_n \geq h_0$.  Additionally, we can assume $\sqrt{(k - 2)^2 + E^2} \leq k$ for $k$ sufficiently large.  As a result, we will analyze three cases, depending on how $h_n$ compares to $\sqrt{(k - 2)^2 + E^2}$ and $k$.
\ \\

\hrule\ 

{\bf Case 1:} Assume $\sqrt{(k - 2)^2 + E^2}, k \leq h_n$.  Then, note that if $y = 1/k$, we can use L'H\^{o}pital's Rule to obtain:
\begin{align*}
\lim_{k \to \infty} k^\alpha - [(k - 2)^2 + E^2]^{\frac{\alpha}{2}} &= \lim_{y \to 0^+} \frac{1 - [(1 - 2y)^2 + y^2E^2]^\frac{\alpha}{2}}{y^\alpha}\\
&= \lim_{y \to 0^+} \frac{-\frac{\alpha}{2}[(1 - 2y)^2 + y^2E^2]^{\frac{\alpha}{2} - 1} [-4(1 - 2y) + 2yE^2]}{\alpha y^{\alpha - 1}}\\
&= \infty.
\end{align*}
This implies that the difference between the two terms can be arbitrarily large as $k$ approaches infinity.  As a result, for $k$ sufficiently large depending on $E$,
\begin{align*}
\phi_n\left(\sqrt{4 + E^2}\right) + \phi_n\left(\sqrt{(k - 2)^2 + E^2}\right) &= (4 + E^2)^\frac{\alpha}{2} + \big((k - 2)^2 + E^2\big)^{\frac{\alpha}{2}}\\
&\leq k^\alpha = \phi_n(k).
\end{align*}
\hrule\ 

{\bf Case 2:} Assume $\sqrt{(k - 2)^2 + E^2} \leq h_n \leq k$.  We break into two subcases.  First, assume $[k - h_n] \alpha h_n^{\alpha - 1} \geq (4+E^2)^{\alpha/2}$.  Then,
\begin{align*}
\phi_n\left(\sqrt{4 + E^2}\right) + \phi_n\left(\sqrt{(k - 2)^2 + E^2}\right) &= (4 + E^2)^\frac{\alpha}{2} + \big((k - 2)^2 + E^2\big)^{\frac{\alpha}{2}}\\
&\leq [k - h_n] \alpha h_n^{\alpha - 1} + h_n^\alpha\\
&= \phi_n(k)
\end{align*}

So, we turn to the other subcase, where $[k - h_n] \alpha h_n^{\alpha - 1} \leq (4 + E^2)^{\alpha/2}$.  We rearrange to obtain:
\[
k - h_n \leq \frac{(4 + E^2)^{\alpha/2}}{\alpha h_n^{\alpha - 1}}.
\]
We note that for $k$ sufficiently large, $\sqrt{(k - 2)^2 + E^2} \leq k - 1$, so that $k - \sqrt{(k - 2)^2 + E^2} \geq 1$.  Combining this with our previous inequality implies:
\[
h_n - \sqrt{(k - 2)^2 + E^2} \geq 1 - \frac{(4 + E^2)^{\alpha/2}}{\alpha h_n^{\alpha - 1}}.
\]
We can use this to bound $\sqrt{(k - 2)^2 + E^2}$, and obtain:
\begin{align*}
\phi_n\left(\sqrt{4 + E^2}\right) + \phi_n\left(\sqrt{(k - 2)^2 + E^2}\right) &= (4 + E^2)^\frac{\alpha}{2} + \big((k - 2)^2 + E^2\big)^{\frac{\alpha}{2}}\\
&\leq (4 + E^2)^\frac{\alpha}{2} + \left[ h_n - 1 + \frac{(4 + E^2)^{\alpha/2}}{\alpha h_n^{\alpha - 1}}\right]^\alpha.
\end{align*}
By choosing $h_0 < h_n$ sufficiently large depending on $E$, this is bounded by:
\[
\left(h_n - \frac{1}{2}\right)^\alpha < h_n^\alpha + \alpha h_n^{\alpha - 1}[k - n] = \phi_n(k).
\]
\hrule\

{\bf Case 3:} Assume $\sqrt{(k - 2)^2 + E^2}, k \geq h_n$.  Then, using again that for $k$ sufficiently large, $k - \sqrt{(k - 2)^2 + E^2} \geq 1$, we see that for $h_0$ sufficiently large,
\[
(4 + E^2)^{\frac{\alpha}{2}} \leq \alpha h_n^{\alpha - 1} \left[k - \sqrt{(k - 2)^2 + E^2}\right],
\]
since $h_0 < h_n$.  Thus, for $k$ and $h_0$ sufficiently large depending on $E$,
\[
(4 + E^2)^{\frac{\alpha}{2}} + h_n^\alpha + \alpha h_n^{\alpha - 1} \big[\sqrt{(k - 2)^2 + E^2} - h_n \big] \leq h_n^\alpha + \alpha h_n^{\alpha - 1}[k - h_n],
\]
which is equivalent to $\phi_n\left(\sqrt{4 + E^2}\right) + \phi_n\left(\sqrt{(k - 2)^2 + E^2}\right) \leq \phi_n(k)$.  This completes the proof.

\end{proof}

\begin{lem} \label{W-dimensions}
Let $e_1$ and $e_2$ be the unit vectors in the first and second coordinate directions of $\mathbb{R}^d$.  There exists $c > 0$ such that for all $\ell \geq 1/2$ and for all $n$,
\[
\ell e_1 \pm c \sqrt{\ell} e_2 \in \mathcal{W}_{\phi_n}(0, 2\ell e_1).
\]
\end{lem}

\begin{proof}
Let $k = c \sqrt{\ell}.$  By the definition of $\mathcal{W}_{\phi_n}$, we need to show the following:
\[
2\phi_n\left( \sqrt{\ell^2 + k^2}\right) \leq \phi_n(2\ell).
\]
Much like the proof of Lemma \ref{E-Regions}, we will break into cases depending on the relationship between $\sqrt{\ell^2 + k^2}, 2\ell$, and $h_n$, the cutoff in the definition of $\phi_n$.  Throughout all of these cases, it will be helpful to recognize that whenever $k/\ell < 1$ we have:
\begin{equation} \label{sqrt-eqn}
\sqrt{\ell^2 + k^2} = \ell \sqrt{1 + \frac{k^2}{\ell^2}} \leq \ell\left(1 + \frac{k^2}{2\ell^2}\right) = \ell + \frac{k^2}{2\ell}.
\end{equation}

\hrule\

{\bf Case 1:}  Assume $2\ell \leq h_n$.  Then, for $c$ sufficiently small, since $k = c\sqrt{\ell}$,
\[
2\phi_{n}\left(\sqrt{\ell^2 + k^2}\right) = 2\left(\sqrt{\ell^2 + k^2} \right)^\alpha \leq 2 \left(\ell + \frac{k^2}{2\ell} \right)^\alpha < 2\left(2^{(\alpha - 1)/\alpha} \ell\right)^\alpha = 2^\alpha\ell^\alpha = \phi_{n}(2\ell),
\]
since we can force $k^2/2\ell \leq (2^{(\alpha - 1)/\alpha} - 1) \ell$ for all $\ell \geq 1/2$ if we pick $c$ in the definition of $k = c\sqrt{\ell}$ sufficiently small.
\ \\

\hrule\ 

{\bf Case 2:}  Now, assume $\sqrt{\ell^2 + k^2} \leq h_n \leq 2 \ell$.  We examine:
\[
\phi_n(2\ell) - 2 \phi_n\left(\sqrt{\ell^2 + k^2}\right) = h_n^\alpha(1 - \alpha) + 2\alpha \ell h_n^{\alpha - 1} - 2\left(\sqrt{\ell^2 + k^2}\right)^\alpha.
\]
Viewing this as a function of $h_n$, we have:
\[
\frac{\partial}{\partial h_n}\left[\phi_n(2\ell) - 2 \phi_n\left(\sqrt{\ell^2 + k^2}\right)\right] = \alpha(1 - \alpha) h_n^{\alpha - 1} + 2\ell\alpha (\alpha - 1)h_n^{\alpha - 2} = \alpha (\alpha - 1) h_n^{\alpha - 2}(2\ell - h_n).
\]
This quantity is non-negative, since $\alpha > 1$ and $2\ell \geq h_n$ by assumption.  Thus, $\phi_n(2\ell) - 2 \phi_n\left(\sqrt{\ell^2 + k^2}\right)$ is minimized in this case when $h_n = \sqrt{\ell^2 + k^2}$.  Thus,
\begin{align}
\phi_n(2\ell) - 2 \phi_n\left(\sqrt{\ell^2 + k^2}\right) &\geq \left(\sqrt{\ell^2 + k^2}\right)^\alpha + \alpha \left(\sqrt{\ell^2 + k^2}\right)^{\alpha - 1} \left[2\ell - \sqrt{\ell^2 + k^2}\right] - 2\left(\sqrt{\ell^2 + k^2}\right)^\alpha \nonumber \\
&= \alpha \left(\sqrt{\ell^2 + k^2}\right)^{\alpha - 1} \left[2\ell - \sqrt{\ell^2 + k^2}\right] - \left(\sqrt{\ell^2 + k^2}\right)^\alpha. \label{Case2-WBound}
\end{align}
Now, rearranging Equation \eqref{sqrt-eqn} gives:
\[
2\ell \geq 2 \sqrt{\ell^2 + k^2} - \frac{k^2}{\ell} > \left(\frac{1}{\alpha} + 1\right) \sqrt{\ell^2 + k^2},
\]
where the last inequality is true for all $\ell \geq 1/2$ as long as $c$ in the definition $k = c\sqrt{\ell}$ is sufficiently small, since $k^2/\ell$ can then be made smaller than $\left(1 - \frac{1}{\alpha}\right) \sqrt{\ell^2 + k^2}$.  This implies that $2\ell - \sqrt{\ell^2 + k^2} \geq \frac{1}{\alpha} \sqrt{\ell^2 + k^2}.$  Plugging this into Equation \eqref{Case2-WBound} gives for $c > 0$ sufficiently small:
\[
\phi_n(2\ell) - 2 \phi_n\left(\sqrt{\ell^2 + k^2}\right) > \alpha \left(\sqrt{\ell^2 + k^2}\right)^{\alpha - 1} \cdot \frac{1}{\alpha} \sqrt{\ell^2 + k^2} - \left(\sqrt{\ell^2 + k^2}\right)^\alpha = 0.
\]
This proves that $2\phi_n\left( \sqrt{\ell^2 + k^2}\right) \leq \phi_n(2\ell)$ in this case.
\ \\

\hrule\ 

{\bf Case 3:}  Finally, assume $\ell^2, \sqrt{\ell^2 + k^2} > h_n$.  Then,
\begin{align}
\phi_n(2\ell) - 2\phi_n\left(\sqrt{\ell^2 + k^2}\right) &= h_n^\alpha + \alpha h_n^{\alpha - 1}(2\ell - h_n) - 2 \left[h_n^\alpha + \alpha h_n^{\alpha - 1}\left(\sqrt{\ell^2 + k^2} - h_n\right) \right] \nonumber \\
&= (\alpha - 1) h_n^\alpha - 2\alpha h_n^{\alpha - 1} \left[\sqrt{\ell^2 + k^2} - \ell\right]. \label{Case3-Wbound}
\end{align}
We want to show that this quantity is positive for $c > 0$ sufficiently small in the definition $k = c\sqrt{\ell}$.  Notice that by Equation \eqref{sqrt-eqn}, we have for all $\ell \geq 1/2$ and for $c$ sufficiently small,
\[
\sqrt{\ell^2 + k^2} - \ell \leq \frac{k^2}{2\ell} < \frac{\alpha - 1}{2\alpha} h_0,
\]
where $h_0$ appears in the definition of $h_n$ and $\phi_n$.  Since $h_0 < h_n$ for all $n$, this condition implies
\[
(\alpha - 1) h_n^\alpha > 2\alpha h_n^{\alpha - 1} \left(\sqrt{\ell^2 + k^2} - \ell\right),
\]
which is exactly what was needed to show that $\phi_n(2\ell) - 2 \phi_n\left(\sqrt{\ell^2 + k^2}\right) > 0$ in Equation \eqref{Case3-Wbound}.  This completes the proof.

\end{proof}

%
%
%
%
\section{Appendix of cited results} \label{Appendix}

\subsection{Results on $T''$ and $T$.}
The following two results from \cite{Howard:2001} (Lemmas 3.1 and 3.3) describe some relationships between $T(0, ne_1)$ and $T''(0, ne_1)$, as well as a connection to an intermediary distance, $T'(0, ne_1)$.  $T'(a, b)$ is defined the same way as $T(a, b)$, except that Poisson points are added at $a$ and $b$. Here, as above, $T$ will be shorthand for $T(0, ne_1)$, with the same convention for $T'$ and $T''$. In order to prove that $T''$ is often equal to $T'$, the authors in \cite{Howard:2001} chose an $\epsilon$ small enough in the definition of $T''$ and $Q_n$ so that $17 \epsilon \sqrt{d}$ was smaller than the critical radius for continuum percolation.
\begin{lem} \label{TDifferences}
For some constants $C_1$ and $c_2 > 0$, we have:
\[
\mathbb{P}\left[|T- T'| > x\right] \leq C_1 \exp\left(-c_2 x^{d/\alpha}\right)
\]
 for all $x > 0$. For $\epsilon > 0$ sufficiently small in the definition of $T''$ and $Q_n$,
\[
\mathbb{P}\left[T' \neq T''\right] \leq C_1 \exp\left(-c_2 n^{1/(2\alpha)}\right).
\]
\end{lem}

\begin{lem} \label{TTails}
Let $\kappa = \min(1, d/\alpha)$.  Then, there exist constants $C_1$ and $c_2 > 0$ such that, for $S = T, T'$, or $T''$, $\mathbb{P}[S > x] \leq C_1 \exp\left(-c_2 x^\kappa \right)$ for all $x \geq C_1 n$.
\end{lem}

Note that the definition of $T''$ depends on $n$, because $T''$ uses the distance function, $\phi_n$, and the point process $Q_n$.  We also need a slightly stronger version of Lemma \ref{TTails}: for any $k \leq n$, we need exponential tails for $T''(0, ke_1)$:

\begin{lem} \label{TprimeTails}
Let $\kappa = \min(1, d/\alpha)$.  For any $k \leq n$, there exist constants $C_1$ and $c_2 > 0$ such that whenever $\ell \geq C_1 k$,
\[
\mathbb{P}(T''(0, ke_1) \geq \ell) \leq C_1 \exp\left(-c_2 \ell^\kappa\right).
\]
\end{lem}

\begin{proof}
Because the definition of $T'$ does not depend on $\ell$ or $n$, Lemma \ref{TTails} above implies this result for $T'$.  Let $q_0, q_1, \ldots, q_N$ be the $T'$-geodesic between $0$ and $ke_1$, where $q_0 = 0$ and $q_N = ke_1$.  Recall that $Q_n$ is the subset of $Q$ in the definition of $T''$.  Note that by definition, there are points $\tilde q_i \in Q_n$ that are within distance $\epsilon/3^\ell \sqrt{d}$ of each $q_i$ in $Q_n$.  Additionally, $\phi_n(x) \leq x^\alpha$ for all $x$.  Beginning by using Lemma \ref{PhiIneq}, we have:
\begin{align*}
\mathbb{P}(T''(0, ke_1) \geq \ell) &\leq \mathbb{P}\left( 2^{2\alpha} \sum_{i = 1}^N \big[ \phi_n( \lVert \tilde q_i - q_i \rVert ) + \phi_n(\lVert q_i - q_{i - 1}\rVert) + \phi_n(\lVert q_{i - 1} - \tilde q_{i - 1}\rVert)\big] \geq \ell\right)\\
&\leq \mathbb{P} \left( \sum_{i = 1}^N \left[\lVert q_i - q_{i - 1}\rVert^\alpha + \frac{2\epsilon \sqrt{d}}{3^n}\right] \geq \ell\right).
\end{align*}
Now, we break into cases: for some constant, $c > 0$, we consider when $N \leq c \ell $ or $N >  c \ell$.  We have for $c$ sufficiently small:
\begin{align*}
\mathbb{P} \left( \sum_{i = 1}^N \left[\lVert q_i - q_{i - 1}\rVert^\alpha + \frac{2\epsilon \sqrt{d}}{3^n}\right] \geq \ell\right) &\leq \mathbb{P} \left( \sum_{i = 1}^N \lVert q_i - q_{i - 1}\rVert^\alpha \geq \ell \left(1 - \frac{2c \epsilon \sqrt{d}}{3^n}\right) \right)\\
&\ \ \   + \mathbb{P}\left( N >  c \ell \right)\\
&\leq \mathbb{P} \left(T'(0, ke_1) > \ell/2\right) + \mathbb{P}\big(N >  c\ell \big).
\end{align*}
By Lemma \ref{TTails}, the first probability has nearly-exponential tails provided that $C_1$ is large enough.  On the other hand, as summarized by Equation (3.21) in \cite{Howard:2001} and the following two equations, large deviation results like those in Section 1.9 of \cite{Durrett:1991} give:
\[
\mathbb{P}\big(N > c \ell \big) \leq C_1 e^{-c_0\ell}
\]
for a sufficiently large choice of $C_1$, a sufficiently small choice of $c_0$, and $\ell > C_1k$.  This completes the proof.
\end{proof}

One of the main advantages of the distance, $T''$, is the modified distance function, $\phi_n$, defined above.  It has a modified triangle inequality, as described in Lemma 5.3 of \cite{Howard:2001}:

\begin{lem} \label{PhiIneq}
For any $a, b, c \in \mathbb{R}^d$ we have
\[
\phi_n^2(|a - c|) \leq 2^{2\alpha}\big(\phi_n^2(|a - b|) + \phi_n^2(|b - c|)\big)
\]
and
\[
\phi_n(|a - c|) - \phi_n(|a - b|) - \phi_n(|b - c|) \leq 2^\alpha h^\alpha.
\]
\end{lem}

Additionally, we will use Lemma 5.4 from \cite{Howard:2001}, which is about the regions $\mathcal{W}_{\phi_n}(a, b)$, defined above, that describe the areas that shorten the path between $a$ and $b$.

\begin{lem} \label{HN:Regions}
For any $E > 0$ and $a, b \in \mathbb{R}^d$, let $\mathcal{H}_E(a, b)$ denote the set
\begin{align*}
\mathcal{H}_E(a, b) = \bigg\{c \in \mathbb{R}^d:\ &\exists \mbox{ a point $p$ on the line segment connecting}\\
& \frac{3}{4} a + \frac{1}{4} b \mbox{ and } \frac{1}{4} a + \frac{3}{4} b \mbox{ such that } \lVert c - p \rVert \leq E\bigg\}.
\end{align*}
Then, for any $E > 0$, there is an $h_0 > 0$ such that $\mathcal{H}_E(a, b) \subset \mathcal{W}_\phi(a, b)$ whenever $\lVert a - b \rVert > h_0$ and $h_1 > h_0$.
\end{lem}

\subsection{Logarithmic Sobolev inequalities}
We need two logarithmic Sobolev inequalities: Theorem 5.1 in \cite{Boucheron:2013} and the theorem from section 8.14 of \cite{Lieb:2001}.  They are reproduced below.
\begin{thm} \label{logSobolevBernoulli}
Let $f: \{-1, 1\}^n \to \mathbb{R}$ be an arbitrary real-valued function defined on the $n$-dimensional binary hypercube, and let $L$ be the uniform measure on $\{-1, 1\}^n$.  Then,
\[
\Ent_L(f^2) \leq \mathbb{E}_L \sum_{i = 1}^n (\Delta_i f)^2,
\]
where $\Delta_i f(X) = f(X_i^+) - f(X_i^-)$, $X_i^+$ is $X$ with a $1$ replacing the $i$th entry of $X$, and $X_i^-$ is $X$ with a $-1$ replacing the $i$th entry of $X$.
\end{thm}
\begin{thm} \label{logSobolevUniform}
Let $f$ be any function where $f \in L^2(\mathbb{R}^d)$ and $\lVert \nabla f \rVert \in L^2(\mathbb{R}^d)$, and let $a > 0$ be any number.  Let $\lVert f \rVert_2$ be the $L^2$ norm of $f$.  Then,
\[
\int_{\mathbb{R}^d} f(x)^2 \ln\left(\frac{|f(x)|^2}{\lVert f \rVert_2^2}\right) dx + n(1 + \ln a) \lVert f \rVert_2^2 \leq \frac{a^2}{\pi} \int_{\mathbb{R}^d} \lVert \nabla f(x) \rVert^2 dx.
\]
\end{thm}

\subsection{Greedy lattice animals}
We will need Theorem 1 from \cite{Gandolfi:1994} in the proofs of Lemma \ref{UniformBound} (before Equations \eqref{pBoundPart1} and \eqref{OtherGreedyLatticeBound}) and Lemma \ref{BernoulliBound} (before Equation \eqref{MPartOne}).  Define a lattice animal to be a face-connected set of unit boxes, where the corners of the boxes are at points in $\mathbb{Z}^d$.  Let $X_\nu$ be any i.i.d. family of non-negative random variables indexed by boxes $\nu$, and let $\mathcal{A}_z(n)$ be the collection of all lattice animals of size $n$ that contain the box centered at $z$. Also, let
\[
M_n = \max_{A \in \mathcal{A}_z(n)} \sum_{\nu \in A} X_\nu.
\]
\begin{thm} \label{GreedyLatticeThm}
If there exists an $a > 0$ such that $\mathbb{E}(X_0^d (\log^+ X_0)^{d + a}) < \infty$, then there exists a constant $M$ such that $M_n / n \to M$ with probability $1$ and in $L^1$.
\end{thm}

\bigskip
\noindent
{\bf Acknowledgements.} The authors thank Xuan Wang for sharing notes on a previous project with M.\,D. on this problem.  The research of M.\,D. is supported by an NSF CAREER grant.  The research of T.\,G. was supported in part by NSF grant DMS-1344199.

\bibliographystyle{alpha}

\bibliography{CombinedBibliography.bib}

\end{document}